\documentclass[11pt,twoside]{article}
\usepackage[margin=1in]{geometry}
\usepackage{amsmath,amsthm,amssymb}
\usepackage{newcent}
\usepackage[utf8]{inputenc}
\usepackage[T1]{fontenc}
\usepackage[colorlinks=true
                                                ,linkcolor=blue
                                                ,citecolor=blue
                                                ]{hyperref}
\usepackage{microtype}
\usepackage{mathtools}
\usepackage{xfrac}
\usepackage{tikz}

\newtheorem{theorem}{Theorem}[section]
\newtheorem*{definition}{Definition}
\newtheorem{lemma}[theorem]{Lemma}

\newcommand\cf{{\em c.f.~}}

\newcommand\ie{{\em i.e., }}

\DeclareMathOperator\ad{ad}

\DeclareMathOperator\Aut{Aut}

\DeclareMathOperator\End{End}

\DeclareMathOperator\id{id}

\DeclareMathOperator\Sym{Sym}

\newcommand\la{\langle}

\newcommand\ra{\rangle}

\newcommand\size[1]{\lvert #1\rvert}
\newcommand\tensor{\otimes}

\newcommand\FF{\mathbb{F}}
\newcommand\kk{\mathbf{k}}
\newcommand\NN{\mathbb{N}}
\newcommand\QQ{\mathbb{Q}}
\newcommand\RR{\mathbb{R}}
\newcommand\ZZ{\mathbb{Z}}

\renewcommand\cal{\mathcal}

\newcommand\al{\alpha}
\newcommand\bt{\beta}
\newcommand\gm{\gamma}                
                
\newcommand\ep{\varepsilon}

\newcommand\kp{\kappa}
\newcommand\lm{\lambda}

\newcommand\sh[1]{(\mathrm #1)}
\newcommand\gap{\vspace{2em}\noindent}

\allowdisplaybreaks
\makeatletter
\renewcommand*\env@matrix[1][\arraystretch]{%
        \edef\arraystretch{#1}%
        \hskip -\arraycolsep
        \let\@ifnextchar\new@ifnextchar
        \array{*\c@MaxMatrixCols c}}
\makeatother
\newcommand{\specialcell}[3][1]{%
        \renewcommand{\arraystretch}{#1}%
        \begin{tabular}{#2}#3\end{tabular}}

\begin{document}
\abovedisplayskip=0.5em
\belowdisplayskip=0.5em

\title{Generalised dihedral subalgebras from the Monster}
\author{F. Rehren}
\date{10th October, 2014}
\maketitle

\begin{abstract}
        The conjugacy classes of the Monster which occur in the McKay observation
        correspond to the isomorphism types
        of certain $2$-generated subalgebras of the Griess algebra.
        Sakuma, Ivanov and others showed that these subalgebras
        match the classification of vertex algebras generated by two Ising conformal vectors,
        or of Majorana algebras generated by two axes.
        In both cases, the eigenvalues $\al,\bt$ parametrising the theory
        are fixed to $\sfrac{1}{4},\sfrac{1}{32}$.
        We generalise these parameters and the algebras which depend on them,
        in particular finding the largest possible (nonassociative) axial algebras
        which satisfy the same key features,
        by working extensively with the underlying rings.
        The resulting algebras admit an associating symmetric bilinear form
        and satisfy the same $6$-transposition property as the Monster;
        $\sfrac{1}{4},\sfrac{1}{32}$ turns out to be distinguished.
\end{abstract}

        The {\em Monster group}, the largest of the sporadic simple groups,
        was first constructed as the automorphism group of the {\em Griess algebra},
        a nonassociative commutative algebra of dimension $196884$.
        Soon an "affinization" $V^\natural$,
        containing (a twisted version of) the Griess algebra in its weight-$2$ subspace,
        was constructed \cite{flm};
        later Borcherd's introduction of {\em vertex algebras} \cite{borcherds}
        unified $V^\natural$ with affine Kac-Moody Lie algebras
        and the two-dimensional conformal field theories of string theory.
        In this way vertex algebras let us view an exceptional object, the Monster,
        as a point in a general theory.

        A key part of the Griess construction
        was that certain idempotents in the algebra
        are in bijection with the involutions in the $\sh{2A}$-conjugacy class of the Monster.
        This also holds in $V^\natural$.
        It was Miyamoto's profound insight
        that every idempotent in the {\em weight-$2$ subspace} of {\em any} (OZ-type) vertex algebra
        give rise to an involution in the automorphism group \cite{miyamoto} of the vertex algebra.
        Idempotents in this context are called {\em $c$-conformal vectors},
        where $c$ is a parameter which connects the idempotent
        to the remarkable representation theory of the infinite-dimensional Virasoro algebra.
        In our case, the $\sfrac{1}{2}$-conformal vectors in $V^\natural$
        induce the involutions of type $\sh{2A}$ in the Monster.

        Sakuma \cite{sakuma} developed Miyamoto's investigation \cite{miyamoto-s3}
        to now show that the product of two involutions induced from any two $\sfrac{1}{2}$-conformal vectors has order at most $6$.
        This implies as a corollary the celebrated fact that the Monster is a $6$-transposition group.
        Furthermore, Sakuma proved that any two $\sfrac{1}{2}$-conformal vectors
        generate a vertex algebra of one of $9$ isomorphism types,
        all occurring in $V^\natural$.
        The same types were also studied by Norton in the Griess algebra \cite{conorton},
        where the isomorphism type is determined
        by the conjugacy class of the product of the corresponding $\sh{2A}$-involutions in the Monster:
        the possibilities are $\sh{1A},\sh{2A},\sh{3A},\sh{4A},\sh{5A},\sh{6A},\sh{4B},\sh{2B},\sh{3C}$.
        Remarkably, these are exactly the classes that McKay observed to label
        the affine $E_8$ Dynkin diagram.

        A radical further step was taken by Ivanov \cite{ivanov}
        in axiomatising some properties of weight-$2$ subspaces of OZ-type vertex algebras generated by $\sfrac{1}{2}$-conformal vectors,
        under the name of {\em Majorana algebras}.
        In this new context Sakuma's theorem still holds \cite{ipss,hrs-sak}
        (now leading to the list of the $9$ finite-dimensional {\em Norton-Sakuma algebras});
        furthermore the Griess algebra is a Majorana algebra,
        and many other small finite groups have been realised
        as automorphism groups of Majorana algebras \cite{akos}.
        This puts the Monster into a family of representations of groups
        on nonassociative commutative algebras.
    The axioms are quite different to those of vertex algebras
        but still includes the parameters $\sfrac{1}{4},\sfrac{1}{32}$
        dictated by the Virasoro algebra at central charge $\sfrac{1}{2}$.

        In this text, we generalise the theory to use the formal parameters $\al,\bt$.
        This enlarges our story to a context introduced in \cite{hrs-sak,hrs}:
        {\em axial algebras} delineate a new class of commutative nonassociative algebras.
        It also allows us to deduce results for $c$-conformal vectors when $c\neq\sfrac{1}{2}$.
        In contrast to well-known nonassociative algebras using word relations,
        such as Lie and Jordan algebras,
        our nonassociative algebras are controlled by {\em fusion rules}.

        \begin{table}[h]
        \begin{center}
        \renewcommand{\arraystretch}{1.5}
        \begin{tabular}{|c||c|c|c|c|}
                \hline
                        \;\;$\star$\;\; & $1$ & $0$ & $\al$ & $\bt$ \\
                \hline\hline
                        $1$ & $\{1\}$ & $\emptyset$ & $\{\al\}$ & $\{\bt\}$ \\
                \hline
                        $0$ &        & $\{0\}$ & $\{\al\}$ & $\{\bt\}$ \\
                \hline
                        $\al$ &        &        & $\{1,0\}$ & $\{\bt\}$ \\
                \hline
                        $\bt$ &        &        &        & $\{1, 0, \al\}$ \\
                \hline
        \end{tabular}
        \caption{Ising fusion rules $\Phi(\al,\bt)$}
                \label{tbl-ising}
        \end{center}
        \end{table}

        The key definition is a {\em $\Phi$-axis}:
        an idempotent, whose eigenspaces multiply according to the fusion rules $\Phi$.
        If structure constants describe the multiplication of elements,
        fusion rules can be understood as describing the multiplication of submodules.
        We study the {\em Ising}\footnote{
            Named after the Ising model in statistical physics,
            whose solution involves the Virasoro algebra at $c=\sfrac{1}{2}$.
        } {\em fusion rules} $\Phi(\al,\bt)$ of Table \ref{tbl-ising}.
        The specialisation to $\Phi(\sfrac{1}{4},\sfrac{1}{32})$
        recovers the Majorana axes and $\sfrac{1}{2}$-conformal vectors.
        Crucially, the $\ZZ/2$-grading on $\Phi(\al,\bt)$
        means that every $\Phi(\al,\bt)$-axis induces an involution in the automorphism group.

        The simpler case of {\em Jordan-type fusion rules} $\Phi(\al)\subseteq\Phi(\al,\bt)$
        was completely solved in \cite{hrs}
        and allows a remarkable new characterisation of $3$-transposition groups
        in the generic case of $\al\neq0,1$;
        in the special case $\al=\sfrac{1}{2}$, it includes some Jordan algebras \cite{tomfelix}.

        We find the axial generalisation of each Norton-Sakuma algebra,
        and thereby introduce new examples of nonassociative algebras with good properties.
        Our results have an algebro-geometric flavour,
        in that to each of our algebras $A_{\al,\bt}$ over a field $\FF$
        we can attach a solution set $\{(\al,\bt)\mid A_{\al,\bt}\text{ exists}\}\subseteq\FF^2$.
        For each Norton-Sakuma algebra we find the irreducible variety of $(\al,\bt)$
        that generalises the algebra.
        This is plotted in Figure \ref{fig-plot}.
        The intersection of all varieties (indeed, of any two not including $\sh{3A'_{\al,\bt}}$)
        is the distinguished point $(\al,\bt) = (\sfrac{1}{4},\sfrac{1}{32})$.

        \begin{figure}[ht]
        \begin{center}
        \begin{tikzpicture}[scale=5.5,domain=-0.2:1.2]
                \fill[red!30] (-0.2,-0.2) rectangle (1.2,1.2);
                \draw[white,thick] (0.5,-0.2) -- (0.5,1.2);
                \draw[red!70] (1.3,0.3) node {$\boldsymbol{\sh{3A'_{\al,\bt}}}$};

                \draw[cyan,thick] (0.25,-0.2) -- (0.25,1.2);
                \draw[cyan] (0.15,0.4) node {$\boldsymbol{\sh{4A_\bt}}$};
                \draw[white,thick] (0.25,0.125) circle[radius=0.001];
                \draw[white,thick] (0.25,0.0625) circle[radius=0.001];

                \draw[blue,thick] plot (\x,\x^2/2);
                \draw[blue] (0.7,0.15) node {$\boldsymbol{\sh{4B_\al}}$};

                \draw[green!80!black,thick,domain=-0.125:1.2] plot (\x,0.625*\x-0.125);
                \draw[green!80!black] (0.8,0.46875) node {$\boldsymbol{\sh{5A_\al}}$};

                \draw[purple,thick] plot[smooth] file {6a-curve1.dat};
                \draw[purple] (0.50,0.9) node[anchor=east] {$\boldsymbol{\sh{6A_\al}}$};
                \draw[white,thick] (0.33333,0.08333) circle[radius=0.001];
                \draw[white,thick] (0.4,0.2) circle[radius=0.001];

                \draw[->,white,thick] (-0.2,0) -- (1.2,0);
                \draw (1.1,0) node[anchor=north] {$\boldsymbol\al$};
                \draw[->,white,thick] (0,-0.2) -- (0,1.2);
                \draw (0,1.1) node[anchor=east] {$\boldsymbol\bt$};

                \draw[white,thick] (-0.2,1) -- (1.2,1);
                \draw[white,thick] (1,-0.2) -- (1,1.2);
                \draw[white,thick] (-0.2,-0.2) -- (1.2,1.2);

                \draw (0,1) node[anchor=east] {$\boldsymbol1$};
                \draw (1,0) node[anchor=north] {$\boldsymbol1$};
                \draw (0.5,0) node[anchor=north] {$\boldsymbol{\sfrac{1}{2}}$};
                \draw (0.8,0.8) node[] {${\al=\bt}$};
                \draw (0.13,0.06) node[] {\footnotesize ${(\sfrac{1}{4},\sfrac{1}{32})}$};
                \draw (0.25,0.03125) node[] {${*}$};
                \draw (0,0) node[anchor=north east] {$\boldsymbol{0}$};

        \end{tikzpicture}
        \end{center}
        \caption{The existence of Norton-Sakuma-like algebras
                for $(\al,\bt)\in\RR^2$}
        \label{fig-plot}
        \end{figure}
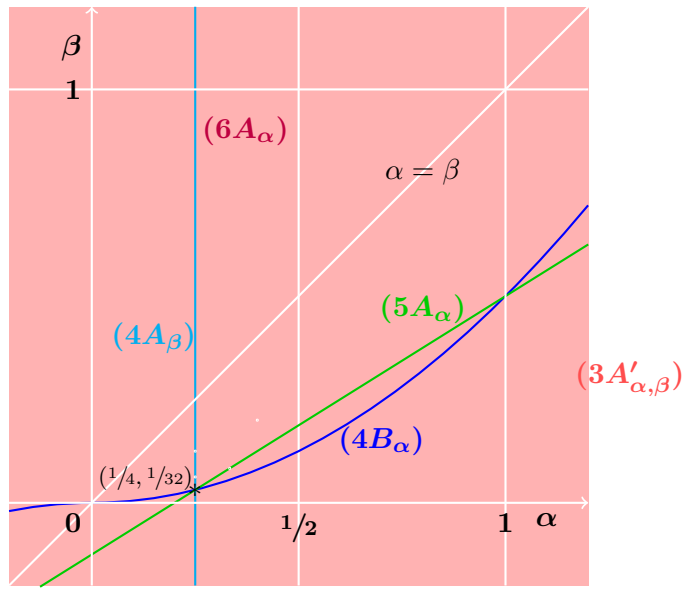

    \gap
        We now give an outline of the paper
        and explain results and methods in more detail.

        In Section \ref{sec-prep},
        we recall our framework for axial algebras,
        which are commutative, nonassociative algebras generated by idempotents
        whose eigenvectors satisfy fusion rules.
        We also introduce Miyamoto involutions,
        the {\em Frobenius} property of having an associating bilinear form,
        and prove several key lemmas.
        We briefly review some results from Majorana theory at the end,
        and results from the Jordan-type fusion rules $\Phi(\al)$.

        The main result of Section \ref{sec-universe}
        is a generalisation of \cite{hrs-sak}'s Theorem 1.1,
        where in particular we no longer require the existence of a bilinear associating form:
        \begin{theorem}[Corollary to Theorem \ref{thm-universe}]
                \label{thm-universal}
                There exists a ring $R$ and a $\Phi$-axial $R$-algebra $A$ such that,
                if $B$ is an $S$-algebra finitely generated by $\Phi$-axes,
                then $S$ is an associative $R$-algebra
                and $B$ is a quotient of $A\tensor_R S$.
        \end{theorem}
        In particular, the classification of quotients of the universal
        $m$-generated $\Phi$-axial algebra
        classifies all the $m$-generated $\Phi$-axial algebras;
        we use this for $m=2$.
        We call a $2$-generated $\Phi$-axial algebra
        a {\em $\Phi$-dihedral algebra} for short.

        In the subsequent Section \ref{sec-mt},
        we compute (with some restrictions on parameters)
        the structure constants for the universal $2$-generated Ising axial algebra.
        Our calculations are more general than,
        but very much inspired by,
        Sakuma's \cite{sakuma};
        we rely on \cite{gap} to accurately complete the large but simple polynomial calculations.
        We do not find all relations among the coefficients,
        but Theorem \ref{thm-mt}
        contains the multiplication on a spanning set of size $8$
        over a finitely-generated polynomial ring.
        Section \ref{sec-fus-rule} introduces the so-called {\em (axial) cover} of a quotient algebra,
        which is its maximal axial generalisation.
        Namely, if $A$ is a simple $2$-generated axial algebra,
        and $U$ is the appropriate universal $2$-generated axial algebra,
        then $A$ is a quotient of $U$ by a maximal ideal $B$;
        the cover of $A$ is the quotient of $U$ by the largest irreducible ideal containing $B$.
        We also record the eigenvectors and a relation from the fusion rules which are key to finding the covers of the Norton-Sakuma algebras.

        In Sections \ref{sec-5A} to \ref{sec-3A},
        under the mild additional assumption of a symmetry between the generators,
        we find covers of the five Norton-Sakuma algebras over $\QQ$
        in which all the permissible eigenvalues $1,0,\sfrac{1}{4},\sfrac{1}{32}$ are realised
        (namely $\sh{3A},\sh{4A},\sh{4B},\sh{5A},\sh{6A}$);
        the covers are dihedral Ising-axial algebras
        which can be specialised to a Norton-Sakuma algebra,
        but it turns out they can also be specialised in infinitely many other ways.
        The degenerate cases $\sh{1A},\sh{2B}$ are their own covers,
        and $\sh{2A},\sh{3C}$ are covered by $\sh{3C_\al}$
        (see also Section \ref{sec-prep}).
        We establish:
        \begin{theorem}[Corollary to the results of Sections \ref{sec-5A} to \ref{sec-3A}]
                \label{thm-covers}
                The (maximal) covers of the Norton-Sakuma algebras $\sh{nX}$ for $n\geq4$
                are given by Table \ref{tbl-covers},
                together with a weak cover of $\sh{3A}$.
                The algebras are Frobenius and satisfy a global $6$-transposition property.
        \end{theorem}

        \begin{table}[h]
        \begin{center}
        \renewcommand{\arraystretch}{1.5}\begin{tabular}{|r||c|c|l|l|c|}
                \hline Alg. & $\size{\rho}$ & dim. & Parameters & Quotients & page\\
                \hline\hline
                $\sh{3A'_{\al,\bt}}$ & $3$ & $4$ &
                        $\al\neq\frac{1}{2}, \quad
                        \lm = \frac{3\al^2+\al(3\bt-1)-2\bt}{4(2\al-1)}$
                        & Lemma \ref{lem-3A-form}
                        & \pageref{tbl-3A} \\
                \hline
                $\sh{4A_\bt}$ & $4$ & $5$ &
                        $\al=\frac{1}{4}, \quad
                        \bt \neq \frac{1}{8},\frac{1}{16}, \quad
                        \lm_1 = \bt, \quad
                        \lm_2 = 0$
                        & $\bt = \frac{1}{2}$
                        & \pageref{tbl-4A} \\
                \hline
                $\sh{4B_\al}$ & $4$ & $5$ &
                        $\al\not\in\eqref{eq-4B-no-kernel}, \quad
                        \bt = \frac{\al^2}{2}, \quad
                        \lm_1 = \frac{\al^2}{4}, \quad
                        \lm_2 = \frac{\al}{2}$ \quad
                        & $\al = -1$
                        & \pageref{tbl-4B} \\
                \hline
                $\sh{5A_\al}$ & $5$ & $6$ &
                        $\bt = \frac{1}{8}(5\al-1), \quad
                        \lm = \frac{3(5\al-1)}{32}$
                        & $\al = -\frac{1}{3},\frac{7}{3}$
                        & \pageref{tbl-5A} \\
                \hline
                $\sh{6A_\al}$ & $6$ & $8$ &
                \!\!\!\specialcell[1.2]{l}{
                        $\al\neq-4\pm2\sqrt5,\frac{1}{3},\frac{2}{5},\frac{1}{2}, \quad
                        \bt = \frac{\al^2}{4(1-2\al)}$ \\
                        $\lm_1=\frac{-\al^2(3\al-2)}{16(2\al-1)^2}, \quad
                        \lm_2=\frac{\al(21\al^2-18\al+4)}{16(2\al-1)^2}$
                }
                        & $\al = \frac{3}{2},\frac{4}{7},\frac{1}{24}(1\pm\sqrt{97})$
                        & \pageref{tbl-6A} \\
                \hline
        \end{tabular}
        \caption{The covers of the Norton-Sakuma algebras}
                \label{tbl-covers}
        \end{center}
        \end{table}

        We further establish the existence of an associating bilinear form on the covers.
        We use this form to determine the $\al,\bt$ for which our algebras have nontrivial quotients over $\RR$.
        The order of the product $\rho$ of the Miyamoto involutions of the two generators is also established.

        I would like to thank S. Shpectorov and C. Hoffman
        for invaluable discussions,
        and the referee for his extensive improvements.

\section{Axial theory}
        \label{sec-prep}

        \begin{definition}
                {\em Fusion rules} are a map $\star\colon\Phi\times\Phi\to2^\Phi$
                on a finite collection $\Phi$ of {\em eigenvalues}.
        \end{definition}
        We often implicitly pair $\star$ with $\Phi$.
        We primarily consider the {\em Ising fusion rules} $\Phi(\al,\bt) = \{1,0,\al,\bt\}$,
        a set of size $4$ with a symmetric map $\star$ given by Table \ref{tbl-ising}.

        Suppose that $R$ is a ring.
        For us, rings are always commutative and associative with $1$.
        An $R$-algebra $A$ is an $R$-module together with a
        commutative but not necessarily associative multiplication.
        An element $a\in A$ is {\em semisimple}
        if $\ad(a)\in\End(A)$ is diagonalisable,
        where $\ad(a)$ is the (left-)adjoint map $x\mapsto ax$;
        equivalently,
        $A$ has a direct sum decomposition into eigenspaces for $\ad(a)$.
        We introduce a special notation for eigenspaces.
        \begin{align}
                A^a_\al & = \{x\in A\mid ax = \al x\}, \\
                A^a_{\{\al_1,\al_2,\dotsc,\al_n\}} & = A^a_{\al_1}\oplus A^a_{\al_2}\oplus\dotsm\oplus A^a_{\al_n}.
        \end{align}
        We likewise write $x^a_\al$
        for the projection of $x\in A$ onto $A^a_\al$,
        and even $x^a_{\{\al_1,\dotsc,\al_n\}} = x^a_{\al_1}+\dotsm+x^a_{\al_n}$.
        We may omit the superscript $a$ if $a$ is understood from context.

        For the most part we will consider rings $R$ which are a field $\FF$,
        a polynomial ring over $\FF$,
        or a quotient of a polynomial ring of $\ZZ$,
        and our algebras are free modules over $R$.
        For this reason we almost always work in the following special situation.
        (We are not aware of similar definitions in the literature.)

        \begin{definition}
                A ring $R$ is {\em everywhere faithful} for its module $M$
                \index{everywhere faithful}
                if, for any nonzero $m\in M$,
                the annihilator ideal $\{ r\in R\mid rm = 0 \}$ of its action
                is always trivial.
        \end{definition}
        
        Such a ring, if $M$ is nonzero, is a domain,
        and contains exactly two idempotents: $0$ and $1$.

        Recall that $a\in A$ is an idempotent if $aa = a$.
        An idempotent $a$ is {\em primitive} if it spans its $1$-eigenspace,
        that is, if $A^a_1 = \la a\ra$.
        In that case, when $R$ is everywhere faithful,
        there exists a map $\lm^a\colon A\to R$ such that $x^a_1 = \lm^a(x)a$.
        This $\lm$ is well-defined,
        since $a$ spans its $1$-eigenspace $A^a_1\ni x^a_1$
        and the $R$-annihilator of $a$ is $0$,
        so $\lm^a(x)\in R$ must be unique to satisfy $x^a_1 = \lm^a(x)a$.

        \begin{definition}
                \label{def axis}
                $a\in A$ is a {\em $\Phi$-axis}
                if $a$ is a semisimple primitive idempotent
                such that
                \begin{equation}
                        A = A^a_\Phi = \bigoplus_{\al\in\Phi}A^a_\al\quad\text{ and }\quad
                        A^a_\al A^a_\bt \subseteq A^a_{\al\star\bt} = \bigoplus_{\mathclap{\gm\in\al\star\bt}}A^a_\gm.
                \end{equation}
        \end{definition}

        A {\em $\Phi$-axial algebra} $A$
        is a commutative nonassociative algebra
        generated by $\Phi$-axes.

        The fusion rules $\Phi$ are {\em $\ZZ/2$-graded}
        if there exists a partition $\Phi = \Phi_+\cup\Phi_-$
        such that, for $\ep,\ep'\in\{+,-\}$ and $\al\in\Phi_\ep,\bt\in\Phi_{\ep'}$,
        $\al\star\bt\subseteq\Phi_{\ep\ep'}$.
        If $\Phi$ is $\ZZ/2$-graded,
        there exists an automorphism $\tau(a)\in\Aut(A)$ for every $\Phi$-axis $a$
        acting as identity on $A^a_{\Phi_+}$ and inverting $A^a_{\Phi_-}$,
        called a {\em Miyamoto involution}.
        On $\Phi(\al,\bt)$ we have a $\ZZ/2$-grading into $\{1,0,\al\}\cup\{\bt\}$.

        \begin{lemma}
                \label{lem-auts-preserve-axes}
                Suppose that $t\in\Aut(A)$
                and that $a\in A$ is a $\Phi$-axis.
                Then $a^t$ is again a $\Phi$-axis.

                If $x\in A$ then $\lm^a(x) = \lm^{a^t}(x^t)$.
                If $\Phi$ is $\ZZ/2$-graded, then $\tau(a)^t = \tau(a^t)$.
        \end{lemma}
        \begin{proof}
                $a^t a^t = (aa)^t = a^t$
                is a nonzero idempotent, and,
                since $t$ is linear and $\ad(a)$ affords a semisimple decomposition,
                $\ad(a^t)=\ad(a)^t$ again affords a semisimple decomposition.
                In particular, if $x\in A^a_{\mu}$,
                then $a^t x^t = (ax)^t = \mu x^t$,
                so $x^t\in A^{a^t}_{\mu}$
                and therefore $(A^a_{\mu})^t \subseteq A^{a^t}_{\mu}$.
                Similarly, if $x\in A^{a^t}_{\mu}$
                then $a^t x = \mu x$ so
                $\mu x^{-t} = (a^t x)^{-t} = ax^{-t}$,
                whence $(A^{a^t}_{\mu})^{-t} \subseteq A^a_{\mu}$,
                \ie $A^{a^t}_{\mu} \subseteq (A^a_{\mu})^t$,
                showing that $A^{a^t}_{\mu} = (A^a_{\mu})^t$
                and furthermore the eigenvalues of $a^t$
                are precisely the eigenvalues of $a$ and lie in $\Phi$.
                Moreover, the fusion rules are also transported:
                suppose that $x\in A^a_{\lm}$ and $y\in A^a_{\mu}$.
                Then $x^t y^t = (xy)^t$
                is in $(A^a_{\lm\star\mu})^t = A^{a^t}_{\lm\star\mu}$, so
                $A^{a^t}_{\mu}A^{a^t}_{\lm}\subseteq A^{a^t}_{\lm\star\mu}$.
                Therefore $a^t$ is a $\Phi$-axis.

                We see that $(A^a_1)^t = A^{a^t}_1$.
                Therefore $\lm^{a^t}(x^t) a^t = (x^t)^{a^t}_1 = (x^a_1)^t = (\lm^a(x) a)^t = \lm^a(x) a^t$.
                If $\Phi$ is $\ZZ/2$-graded, then $t$
                interchanges the $-1$-eigenspaces of $\tau(a)$ and $\tau(a^t)$.
        \end{proof}

        If $a_0,a_1$ are $\Phi$-axes and $\Phi$ is $\ZZ/2$-graded,
        we write $T = \la\tau(a_0),\tau(a_1)\ra$ and $\rho = \tau(a_0)\tau(a_1)$.
        Then $T$ is a dihedral group and we set
        \begin{equation}
                a_{2i} = {a_0}^{\rho^i},\quad
                a_{2i+1} = {a_1}^{\rho^i}.
        \end{equation}

        The following important result also implies Lemma 4.1 in \cite{sakuma}:
        \begin{lemma}
                \label{lem-global-order}
                Suppose that $\Phi$ are $\ZZ/2$-graded fusion rules,
                and $a_0,a_1$ are $\Phi$-axes generating a subalgebra $B$ of $A$.
                Then $\size{a_0^T}=\size{a_1^T}$
                and $\rho^{\size{a_0^T\cup a_1^T}} = 1$ as an automorphism of $A$.
        \end{lemma}
        \begin{proof}
                The conjugacy classes of two generating involutions in a dihedral group
                have equal size, so $\size{\tau(a_0)^T} = \size{\tau(a_1)^T}$.
                Furthermore $\size{\rho} = \size{\tau(a_0)^T\cup\tau(a_1)^T}$.
                As $\tau(a_i)^T = \{\tau(a)\mid a\in a_i^T\}$
                we also have $\size{\tau(a_i)^T}\leq\size{a_i^T}$.
                Let $\cal A = \{a_i\mid i\in\ZZ\} = a_0^T\cup a_1^T$
                and $n$ the smallest positive integer such that $a_0 = a_n$.
                If no such $n$ exists then $\cal A$ is infinite
                and both $a_0^T$ and $a_1^T$ are infinite
                (as $\rho$ has infinite order,
                one of $\tau(a_0)^T$ and $\tau(a_1)^T$ must be infinite,
                and the sizes of these orbits coincides,
                so both are infinite).
                Suppose instead that there is such an $n$.

                If $n$ is odd then
                $n = 2m+1$ and $a_0 = a_{2m+1} = a_1^{\tau^m}$.
                Therefore $a_0^T = a_1^T$ has size $n$ and
                \begin{equation}
                        \tau(a_1) = \tau(a_0)^{\rho^{-m}} = \rho^{m}\tau(a_0)\rho^{-m} = \tau(a_0)\rho^{-2m},
                \end{equation}
                so $\rho^{1+2m} = \rho^n = 1$.
                In the last step, we used that $\rho^{\tau(a_0)} = \rho^{-1}$.

                If $n$ is even then
                $n = 2m$ and $a_0 = a_0^{\rho^m}$.
                Therefore
                \begin{equation}
                        \tau(a_0) = \rho^{-m}\tau(a_0)\rho^{m} = \tau(a_0)\rho^{2m},
                \end{equation}
                so $\rho$ has order $2m$ as required.
                On the other hand, $\tau(a_0)^T$ has size at most $m$,
                so in fact both $\tau(a_0)^T$ and $\tau(a_1)^T$ have size $m$
                and are disjoint.
                We see that $\tau(a_0)^T$ has the same cardinality as $a_0^T$,
                so repeating the argument with $a_1$ in place of $a_0$
                shows that $\tau(a_1)^T$ and $a_1^T$ are also of equal size.
                Therefore $a_1^T$ and $a_0^T$ have size $m$ and are disjoint.
        \end{proof}

        A fusion rule $\Phi$ is {\em Seress}
        if for all $\al\in\Phi$ we have that $1\star\al \subseteq \{\al\} \supseteq 0\star\al$ (and in particular $1\star0 = \emptyset$);
        $\Phi(\al,\bt)$ is Seress.
        \begin{lemma}
                \label{lem seress}
                If $\Phi$ is Seress
                and $a\in A$ is a $\Phi$-axis,
                then for any elements $x\in A,z\in A^a_{\{1,0\}}$
                we have $a(xz) = (ax)z$.
        \end{lemma}
        \begin{proof}
                By linearity, we may suppose that $x\in A^a_\al$.
                As $x,xz\in A^a_\al$, $a(xz) = \al xz = (ax)z$.
        \end{proof}

        Finally, we say that an algebra $A$ is {\em Frobenius}
        if there exists an bilinear form $(,)$ on $A$
        which is {\em associating}: for all $x,y,z\in A$,
        $(xy,z) = (x,yz)$.
        Note that the eigenspaces of an element $a\in A$
        are $(,)$-perpendicular if the pairwise difference of their eigenvalues is invertible,
        and $(,)$ is symmetric if $A$ is generated by idempotents.

        The Jordan-type fusion rules are the restriction $\Phi(\al)$ of $\Phi(\al,\bt)$ to the set $\{1,0,\al\}$.
        From \cite{hrs} we deduce
        \begin{theorem}[follows from \cite{hrs}, Theorem 1.1]
                \label{thm-hrs}
                If $\kk$ is a field containing $\frac{1}{2}$ and $\bar\al\neq0,1$,
                and $A$ is a $2$-generated $\Phi(\bar\al)$-axial $\kk$-algebra,
                then $A$ is a scalar extension of a quotient of the algebra $\hat A$
                with basis $\{a,b,s\}$
                and multiplication from Table \ref{tbl-3C}
                over the ring
                \begin{equation}
                        R = \ZZ[\sfrac{1}{2},\al,\lm]/(2\al-1)(\al-2\lm),
                \end{equation}
                such that the images of $\al,\lm$ are $\bar\al,\lm^a(b)$ respectively.
                \qed
        \end{theorem}

        \begin{table}[h]
        \begin{center}
        \renewcommand{\arraystretch}{1.5}
        \begin{tabular}{|c||c|c|c|}
                \hline
                        $\sh{3C_\al}$ & $a$ & $b$ & $s$ \\
                \hline\hline
                        $a$ & \;\;$a$\;\; & $s + \al(a+b)$ & $((1-\al)\lm-\al)a$ \\
                \hline
                        $b$ &        & $b$ & $((1-\al)\lm-\al)b$ \\
                \hline
                        $s$ &        &        & $((1-\al)\lm-\al)s$ \\
                \hline
        \end{tabular}
        \caption{The algebra $\sh{3C_\al}$}
                \label{tbl-3C}
        \end{center}
        \end{table}

        In fact, if $A$ is a $2$-generated $\Phi(\bar\al)$-axial $\kk$-algebra
        then there exist ideals $\hat J_A\subseteq \kk[\al,\al^{-1},(1-\al)^{-1},\lm]$
        and $\hat J_A\hat A\subseteq\hat I_A\subseteq\hat A\tensor_\ZZ\kk$
        whose quotient is $A$,
        and $\hat J_A$ contains either $(2\al-1)$ or $(\al-2\lm)$.
        Accordingly, set the ideal $\bar J_A$ to be $(2\al-1)$ or $(\al-2\lm)$.
        This $\bar J_A$ is prime,
        and together with an ideal $\bar I_A$, affords the {\em cover} of $A$;
        for example, $\bar J_{\sh{2A}} = \bar J_{\sh{3C}} = (\al-2\lm)$.
        See also Section \ref{sec-fus-rule}.

        The following was proven in \cite{hrs-sak},
        following the groundbreaking work in \cite{sakuma,ipss}.
        \begin{theorem}[Theorems 5.10, 8.7 \cite{hrs-sak}]
                \label{thm-hrs-sak}
                There exists an algebra $U$ over $\QQ^9$
                such that any $2$-generated $\Phi(\sfrac{1}{4},\sfrac{1}{32})$-axial Frobenius algebra over $\QQ$
                is a quotient of $U$,
                and $U$ is the direct sum of the Norton-Sakuma algebras over $\QQ$.
                \qed
        \end{theorem}

 \section{The universal algebra}
         \label{sec-universe}
 
         In this section, we make the formal construction
         of a certain universal algebra:
         an algebra of which all $\Phi$-axial algebras,
         or in our case of interest all $\Phi(\al,\bt)$-dihedral algebras,
         are quotients.
         \index{axial algebra!universal}
         We proceed by constructing a chain of increasingly specialised universal objects,
         starting from a free magma.
 
         Let $\{\al_3,\dotsc,\al_n\}$ be a collection of symbols,
         $\al_1 = 1,\al_2 = 0$,
         and set $\Phi = \{1,0\}\cup\{\al_3,\dotsc,\al_n\}$.
         For $R$ a ring, recall that $R[x,x^{-1}]$ means $R[x,y]/(xy-1)$.
         Let
         \begin{equation}
                 \label{eq-r''}
                 R'' = R''(\Phi) = \ZZ[\al_i,(\al_i-\al_j)^{-1}\mid \al_i,\al_j\in\Phi,i\neq j].
         \end{equation}
 
         Suppose that $R$ is a ring
         and $S$ is an associative $R$-algebra.
         For an $R$-algebra $A$,
         the {\em scalar extension by $S$} of $A$
         is the $S$-algebra $A\tensor_R S$
         with product $(x\tensor s)(y\tensor s') = (xy)\tensor(ss')$.
 
         Let $\cal A = \{a_1,\dotsc,a_m\}$
         be an ordered collection
         and $\cal M'$ the nonassociative magma on $\cal A$,
         that is, the collection of all bracketings of nonempty words on $\cal A$
         together with a multiplication given by juxtaposition.
         In the category of $R$-algebras with $m$ marked generators,
         where morphisms are algebra homomorphisms
         mapping the marked generators to marked generators and preserving the ordering,
         $R\cal M'$ is an initial object:
         there exists exactly one morphism from $R\cal M'$
         to any other object $A$ in the category.
         This mapping $R\cal M'\to A$
         is given by evaluating the word $w\in\cal M'$ as a word in $A$.
 
         Furthermore set $\cal M$ to be the commutative nonassociative magma
         on idempotents $\cal A$,
         that is, $\cal M$ is $\cal M'$ modulo the relations
         $a_ia_i = a_i$ for all $i$ and $uv = vu$ for all words $u,v$.
         Then $R\cal M$ is an initial object
         in the category of commutative (nonassociative) $R$-algebras
         generated by $m$ marked idempotents.
 
         Let $\lm^a(w)$ be a symbol for all $a\in\cal A$ and $w\in \cal M$,
         and set
         \begin{equation}
                 \label{eq-lm-M}
                 R' = R''[\lm^{a}(w)\mid a\in\cal A,w\in \cal M] = R''[\lm^{\cal A}(\cal M)].
         \end{equation}
 
         \begin{lemma}
                 \label{lem-struc-consts}
                 Suppose that $A$ is an everywhere faithful $R$-algebra,
                 generated by a set $\cal A$ of primitive idempotents
                 with eigenvalues $\Phi$.
                 If $R$ is an associative $R''=R''(\Phi)$-algebra,
                 then $R$ is an associative $R'$-algebra in a unique way
                 such that $w^a_1 = \lm^a(w)a$ for all $a\in\cal A,w\in\cal M$.
         \end{lemma}
         \begin{proof}
                 We use the multiplication in $A$
                 to identify the mapping of $\lm^a(w)\in R'$ into $R$.
                 Let $\hat R$ be the polynomial ring $R[\lm^{\cal A}(\cal M)]$ of $R$
                 extended by indeterminates $\lm^a(w)$
                 for all $a\in\cal A,w\in\cal M$ (\cf\eqref{eq-lm-M}).
                 Then there exists a unique mapping of $R'$ into $\hat R$:
                 as $R$ is an associative $R''$-algebra,
                 there exists $\tilde\psi\colon R''\to R$;
                 this extends canonically to $\hat\psi\colon R''[\lm^{\cal A}(\cal M)]\to R[\lm^{\cal A}(\cal M)]$,
                 that is, $\hat\psi\colon R'\to\hat R$,
                 by setting $\tilde\psi(\lm^a(w)) = \lm^a(w)$.
                 Let $J$ be the ideal of $\hat R$
                 generated by $\lm^a(w)-\lm^a(\tilde w)$
                 for all $a\in\cal A$ and $w\in \cal M$,
                 where $\tilde w\in A$ is the evaluation of the word $w\in \cal M$
                 and the second $\lm^a$ is the ordinary mapping $A\to R\subseteq\hat R$.
                 Then set $\psi\colon R'\to R$
                 to be the map $r\mapsto \hat\psi(r)/J$.
                 As $\tilde\psi,\hat\psi$ are ring homomorphisms,
                 so is $\psi$,
                 and $\psi$ makes $R$ an associative $R'$-algebra.
         \end{proof}
 
         Let $A'$ be the quotient of $R'\cal M$ modulo the ideal generated by
         \begin{gather}
                 \label{eq-m-ideal}
                 \Bigl(\prod_{\mathclap{\al_i\in\Phi\smallsetminus\{1\}}}(\ad(a)-\al_i \id)\Bigr)(w-\lm^a(w)a)
                 \quad\text{ for all }a\in\cal A,w\in A''.
         \end{gather}
         Then $A'$ is generated by primitive {\em diagonalisable} idempotents.
 
         \begin{lemma}
                 \label{lem-c''-initial}
                 There exists a unique morphism from $A'$
                 to any $R'$-algebra $A$
                 generated by $m$ primitive diagonalisable idempotents with eigenvalues $\Phi$.
         \end{lemma}
         \begin{proof}
                 As in the proof of Lemma \ref{lem-struc-consts},
                for $w\in A''$, write $\tilde w$ for the evaluation of the word $w$ in $A$,
                and $\bar w$ for the evaluation of the word $\tilde w$ in $A''$.
                Set $I_A$ to be the ideal generated by $w - \bar w$ for $w\in A''$.
                Then $I_A$ is the evaluation ideal in $A''$ induced by $A$,
                and is the unique mapping from $A''$ to $A$
                preserving the $m$ generating idempotents.
                That the $m$ generators in $A$ are primitive diagonalisable idempotents
                 with eigenvalues $\Phi$
                 implies that the ideal in~\eqref{eq-m-ideal}
                 is contained in $I_A$
                 and therefore the unique mapping from $A''$ to $A$
                 factors through a morphism from $A'$ to $A$.
         \end{proof}
 
         Let $\cal C'=\cal C'_{R''}$ be the category
         whose objects are pairs $(R,A)$
         such that $R$ is a ring and an associative $R''$-algebra,
         the generators $\cal A$ in $A$ are primitive diagonalisable idempotents
         with (ordered) eigenvalues $\bar\Phi=\{1,0,\bar\al_2,\dotsc,\bar\al_n\}$,
         and $A$ is an everywhere faithful $R$-algebra generated by $\cal A$.
         (For convenience, we always write $\cal A$ for the generating set,
         as the generators can be identified without danger of confusion.)
         Morphisms in the category from $(R_1,A_1)$ to $(R_2,A_2)$
         are pairs $(\phi,\psi)$
         such that $\phi\colon R_1\to R_2$ is a $R''$-algebra homomorphism
         and $\psi\colon A_1\to A_2$ is an algebra homomorphism
         mapping generators to generators and compatible with $\phi$,
         in the sense that $(rx)^\psi = r^\phi x^\psi$
         for all $r\in R_1,x\in A_1$.
 
         \begin{lemma}
                 \label{lem-c'-initial}
                 $(R',A')$ is an initial object in $\cal C'$.
         \end{lemma}
         \begin{proof}
                 That is, for any $(R,A)$ in $\cal C'$
                 there exists a single morphism $(\phi,\psi)\colon(R',A')\to(R,A)$.
                 This follows since $R$ is an $R'$-algebra,
                 by Lemma~\ref{lem-struc-consts},
                 by a unique nontrivial map $\phi\colon R'\to R$
                 which maps $R'\mapsto (R'\tensor_\ZZ R)/I_R$,
                 where $I_R$ is the ideal $(\al_i-\bar\al_i\mid 3\leq i\leq n)$
                 encoding the identification of $\bar\al_i$ with $\al_i$.
                 Furthermore $A'$ is initial among $R'$-algebras generated by $m$ marked primitive idempotents
                 acting diagonalisably with eigenvalues $\Phi$
                 by Lemma~\ref{lem-c''-initial}.
         \end{proof}
 
         Suppose that $(J,I)$ is a pair
         with $J$ an ideal of $R$ and $I$ an ideal of $A$.
         The pair $(J,I)$ {\em match} if $JA\subseteq I$
         and $\lm^a(x)\in J$ for all $x\in I,a\in\cal A$.
         If $(J,I)$ match and $I$ contains none of the generators in $\cal A$,
         then $(R/J,A/I)$ is again an element of $\cal C'$,
         since $\cal A$ is preserved,
         $R/J$ is again an associative algebra over $R''$ (and hence $R'$)
         acting everywhere faithfully on $A/I$,
         and $A/I$ is generated by $m$ primitive diagonalisable idempotents with eigenvalues $\Phi$.
 
         Take an arbitrary object $(R,A)$
         and let $(\phi,\psi)$ be the morphism from $(R',A')$ to $(R,A)$.
         Then $\psi$ is a surjection in the sense that $R{A'}^{\psi} = A$.
         We can write $\phi$ as the composition $\phi_s\circ\phi_i$
         of an injection $R'\mapsto R'\tensor_\ZZ R$
         with a surjection $(R'\tensor_\ZZ R)/J_A$,
         for $J_A = \ker\phi_s$ an ideal.
         Let $I_A = \ker\psi$.
         The ideals $(J_R,I_A)$ match
         by the condition that $\psi,\phi_s$ have to be compatible.
         Write $\cal C'(J,I)$ for the subcategory of $\cal C'$
         whose objects are algebras $(R,A)$
         such that $J_R\subseteq J,I_A\subseteq I$.
         In $\cal C'(J,I)$, the object $(R'/J,A'/I)$ is again initial.
 
         So far, $\Phi$ has only been a collection of eigenvalues.
         Now suppose that $\Phi$ comes with fusion rules $\star$.
         We write down two special ideals $J',I'$:
         let $I'$ be generated by
         \begin{equation}
                 \prod_{\mathclap{\al\in\al_j\star\al_k}}(\ad(a_i)-\al\id)\Bigl(
                 \prod_{\mathclap{\al\in\Phi\smallsetminus\{\al_j\}}}(\ad(a_i)-\al\id)x\cdot
                 \prod_{\mathclap{\al\in\Phi\smallsetminus\{\al_k\}}}(\ad(a_i)-\al\id)y\Bigr)
         \end{equation}
         for all $x,y\in A$ and $\al_j,\al_k\in\Phi$ and $a_i\in\cal A$,
         and set $J' = (\lm^a(x)\mid a\in\cal A,x\in I')$.
         Then $(J',I')$ match.
         The ideals $J',I'$ are minimal
         with respect to the axioms of $\Phi$-axial algebras,
         that is, if $(R,A)$ are $\Phi$-axial $R'$-algebras
         then they are in $C'$
         and the unique morphism $(\phi,\psi)\colon(R',A')\to(R,A)$
         necessarily has $J'\subseteq\ker\phi$ and $I'\subseteq\ker\psi$.
         Therefore $\cal C = \cal C'(J',I')$
         is the category of $m$-generated $\Phi$-axial algebras
         and has initial object $(R_U,U) = (R'/J',A'/I')$.
         We have
 
         \begin{theorem}
                 \label{thm-universe}
                 For fusion rules $\Phi=\{1,0,\al_3,\dotsc,\al_n\}$,
                 $m\in\NN$,
                 and $R''$ from~\eqref{eq-r''},
                 there exists an algebra $U$ over a ring $R_U$
                 such that if $S$ is an associative $R''$-algebra
                 and $B$ is an $m$-generated $\Phi$-axial $S$-algebra,
                 there exist matching ideals
                 $J_B\subseteq R_U \tensor_\ZZ S,I_B\subseteq U\tensor_\ZZ S$
                 such that $U\tensor_\ZZ S/I_B\cong B$
                 as an algebra over $R_U\tensor_\ZZ S/J_B\cong S$.
                 \qed
         \end{theorem}
 
         We say that the $R_U$-algebra $U$
         is the {\em universal $m$-generated $\Phi$-axial algebra}.
         This implies Theorem \ref{thm-universal}:
         finite generation means that there exists an $m$ for which Theorem \ref{thm-universe} gives the desired result.

         Note that, if we replaced the underlying ring $R''$ in \eqref{eq-r''}
         with a larger ring $\hat R''$,
         Theorem~\ref{thm-universe} continues to hold with minor modification.
         In particular, set
         \begin{equation}
                 \label{eq r0}
                 R_0 = \ZZ[\sfrac{1}{2},\al,\bt,
                \al^{-1},\bt^{-1},(\al-\bt)^{-1},(\al-2\bt)^{-1},(\al-4\bt)^{-1},
                \lm_1,\lm_1^f,\lm_2,\lm_2^f].
         \end{equation}
         The upcoming Theorem~\ref{thm-mt} gives the multiplication table
         for an algebra $A_{R_0}$ over $R_0$.
         The universal $\Phi(\al,\bt)$-dihedral algebra $U$ that exists by Theorem~\ref{thm-universe}
         is a quotient of $A_{R_0}$,
        although we will not completely determine $U$ in this text.

\section{The multiplication table}
        \label{sec-mt}

        From now on,
        we only consider the Ising fusion rules $\Phi = \Phi(\al,\bt)$.
        Suppose that $R$ is an associative $R''(\Phi)$-algebra from \eqref{eq-r''},
        so that $\frac{1}{2},\al,\bt\in R$
        and $\al,\bt,\al-1,\bt-1,\al-\bt$ are invertible in $R$.
        Suppose that $A$ is a (commutative, nonassociative) $R$-algebra
        generated by $\Phi$-axes $a_0,a_1$.
        In this section, culminating in Theorem \ref{thm-mt},
        we give a spanning set and multiplication table
        for the algebra structure of $A$,
        under some further conditions on $R$.

        Multiplication in $A$ is denoted by juxtaposition.
        A second product $\circ\colon A\times A\to A$ on $A$,
        \begin{equation}
                x\circ y = xy - \bt(x+y) \text{ for } x,y\in A,
        \end{equation}
        will play a critical role for us because of the $\ZZ/2$-grading on $\Phi$.

        Recall that, for a $\Phi$-axis $a\in A$,
        $\tau(a)\in\Aut(A)$ is the automorphism
        acting as $-1$ on the $\bt$-eigenspace of $a$
        and fixing everything else,
        and our conventions from Section \ref{sec-prep}.

        Let $B$ be the subset of $A$ with elements
        \begin{gather}
                B = \{ a_{-2},a_{-1},a_0,a_1,a_2,s,s_2,s_2^f \}, \\
                \text{ for } s = a_0\circ a_1,\quad
                s_2 = a_0\circ a_2,\quad
                s_2^f = a_{-1}\circ a_1.
        \end{gather}
        In this section we compute the products between elements of $B$,
        recorded in lemmas;
        other computations will be part of the rolling text.

        Four elements of the ring will also have a special role to play.  Namely, we write
        \begin{equation}
                \lm_1 = \lm^{a_0}(a_1),\quad
                \lm_1^f = \lm^{a_1}(a_0),\quad
                \lm_2 = \lm^{a_0}(a_2),\quad
                \lm_2^f = \lm^{a_1}(a_{-1}).
        \end{equation}
        Note that, since $\al,\bt,\lm,\lm^f,\lm_2,\lm_2^f$ lie in the ring $R$,
        the automorphisms of $A$ and in particular $\tau(a_0),\tau(a_1)$ fix them.

        The superscript $f$ notation
        refers to the map $f$, called the {\em flip},
        interchanging $a_0,a_1$.
        In general $f$ is not an automorphism,
        but it turns out to be in some special cases.
        It has an action on the ring and on the algebra.

        \gap
        Recall that, by the definition of a $\Phi$-axis $a$,
        any element $x\in A$ may be written as
        \begin{equation}
                x = x^{a}_1 + x^{a}_0 + x^{a}_\al + x^{a}_\bt
                 = x_1 + x_0 + x_\al + x_\bt,
        \end{equation}
        where $x^a_\phi$ is the projection of $x$ onto the $\phi$-eigenspace $A^a_\phi$ of $a$.
        We omit the superscript $a$ when context makes it clear
        which idempotent $a$ is being used for this decomposition.
        In this section, all eigenvector decompositions
        will be with respect to $a$ or $a_0$ unless indicated otherwise.
        Furthermore, recall that $x_{1,0}$ is the projection of $x$ onto $A^a_1\oplus A^a_0$,
        and $x_+$ is the projection of $x$ onto $A^a_1\oplus A^a_0\oplus A^a_\al$.
        
        We deduce that
        \begin{align}
                \label{eq-x-al}
                x_\al & = \frac{1}{\al}({a}x - \lm^{a}(x){a} - \bt x_\bt), \\
                \label{eq-x-bt}
                x_\bt & = \frac{1}{2}(x - x^{\tau(a)}).
        \end{align}
        The latter follows since $\tau(a)$ inverts the $\bt$-eigenspace.
        So we have that $x_\bt = \frac{1}{2}(x - x^{\tau(a)})$.
        The former is found by using the previous equations in
        ${a}x = \lm^{a}(x){a} + \al x_\al + \bt x_\bt$ and rearranging.
        Furthermore
        \begin{equation}
                \label{eq aax}
                a^2x = {a}({a}x) = \lm^a(x)(1-\al){a} + \al {a}x + \bt(\bt-\al)x_\bt,
        \end{equation}
        using the substitution \eqref{eq-x-al}
        in the expression $a^2x=\lm^{a}(x)a + \al^2 x_\al + \bt^2 x_\bt$.

        It follows that $a\circ x$ lies in $A^{a}_{\{1,0,\al\}}$
        (and thus is fixed by $\tau(a)$)
        for any $x\in A$:
        \begin{align}
                \label{eq cdot al}
                a\circ x & = \lm^{a}(x)a + \al x_\al + \bt x_\bt - \bt(x+a), \\
                & = (\lm^{a}(x)-\bt)a - \bt x_0 + (\al-\bt) x_\al.
        \end{align}

        \begin{lemma}
                \label{lem-a(acb)}
                We deduce
                \begin{equation}
                        \begin{aligned}
                                a_0s & = (\al-\bt)s + (\lm(1-\al)+\bt(\al-\bt-1)){a_0} + \frac{1}{2}(\al-\bt)\bt\bigl(a_1+a_1^{\tau({a_0})}\bigr),\\
                                a_1s & = (\al-\bt)s + (\lm^f(1-\al)+\bt(\al-\bt-1)){a_1} + \frac{1}{2}(\al-\bt)\bt\bigl(a_0+a_0^{\tau({a_1})}\bigr),\\
                                a_0s_2 & = (\al-\bt)s_2 + (\lm_2(1-\al)+\bt(\al-\bt-1)){a_0} + \frac{1}{2}(\al-\bt)\bt\bigl(a_2+a_2^{\tau({a_0})}\bigr)
                        \end{aligned}
                \end{equation}
                from
                \begin{equation}
                        \label{eq a(acx)}
                        a(a\circ x) = (\al-\bt)a\circ x + (\lm(x)(1-\al)+\bt(\al-\bt-1))a + (\al-\bt)\bt x_+.
                \end{equation}
        \end{lemma}
        \begin{proof}
                The equation \eqref{eq a(acx)} can be deduced
                starting from the definition of $\circ$:
                \begin{align*}
                        a(a\circ x) & = a(ax) - \bt(aa + ax) \\
                                & = (\al-\bt)ax + (\lm(x)(1-\al)-\bt)a + \bt(\bt-\al)x_\bt
                        \intertext{by \eqref{eq aax},
                        and then rewriting $ax$ using $\circ$ as}
                                & = (\al-\bt)a\circ x + (\al-\bt)\bt(a + x) + (\lm(x)(1-\al)-\bt)a - (\al-\bt)\bt x_\bt.
                \end{align*}
                This gives the result after collecting terms
                and writing $x_+ = x - x_\bt$.
                From \eqref{eq a(acx)} we deduce $a_0s,a_0s_2$
                by substituting $a_1,a_2$ for $x$ respectively.
                Then $a_1s$ follows by swapping $a_0,a_1$.
        \end{proof}

        \gap
        We define a further commutative product
        \begin{equation}
                *_a\colon A\times A \to A,\quad
                x*_ay = (x\circ a)y + (y\circ a)x
        \end{equation}
        for our computations.
        With respect to the $\Phi$-axis $a$, for $x,y\in A$,
        \begin{align}
                \label{eq star 10}
                (x*_ay)_{1,0} & =
                        -2\bt (x_+y_+)_{1,0} + 2\al x_\al y_\al
                        + ((\lm^{a}(x)-\bt)\lm^{a}(y)+(\lm^{a}(y)-\bt)\lm^{a}(x))a, \\
                \label{eq star al}
                (x*_ay)_\al & =
                        (\al-2\bt)(x_+y_+)_\al
                        + \al(\lm^{a}(x)-\bt)y_\al + \al(\lm^{a}(y)-\bt)x_\al.
        \end{align}
        To deduce this, we use the previous calculations
        for $(x\circ a)_{1,0}$ and $(x\circ a)_\al$.
        From the fusion rules, since $(x\circ a)_\bt = 0$,
        we have that
        \begin{align*}
                (x*_ay)_{1,0}
                & = ((x\circ a)y)_{1,0} + ((y\circ a)x)_{1,0} \\
                & = (x\circ a)_{1,0}y_{1,0} + (y\circ a)_{1,0}x_{1,0}
                                + (x\circ a)_\al y_\al + (y\circ a)_\al x_\al \\
                & = -2\bt x_{1,0}y_{1,0} + 2(\al-\bt)x_\al y_\al
                        + ((\lm^{a}(x)-\bt)\lm^{a}(y)+(\lm^{a}(y)-\bt)\lm^{a}(x))a,
        \end{align*}
        which gives \eqref{eq star 10} when we use that
        $(x_+y_+)_{1,0} = x_{1,0}y_{1,0}+x_\al y_\al$.
        For \eqref{eq star al},
        \begin{align*}
                (x*_ay)_\al & = ((x\circ a)y)_\al + ((y\circ a)x)_\al \\
                        & = (x\circ a)_{1,0}y_\al + (x\circ a)_\al y_{1,0}
                                + (y\circ a)_{1,0}x_\al + (y\circ a)_\al x_{1,0} \\
                        & = (\al-2\bt)(x_{1,0}y_\al + x_\al y_{1,0})
                                + \al(\lm^{a}(x)-\bt)y_\al + \al(\lm^{a}(y)-\bt)x_\al,
        \end{align*}
        and after using $x_{1,0}y_\al + x_\al y_{1,0} = (x_+y_+)_\al$
        we have the answer.

        We also find that, for $x,y\in A$,
        when $\al-2\bt$ is invertible,
        \begin{align}
                \label{eq-++al}
                (x_+y_+)_\al & = \frac{1}{\al-2\bt}\bigl((x*_ay)_\al+\al\bt(x_\al + y_\al)-\al(\lm^{a}(x)y_\al + \lm^{a}(y)x_\al)\bigr), \\
                \label{eq-alal}
                x_\al y_\al & = \frac{1}{2\al}\bigl( (x*_ay)_{1,0} + 2\bt (x_+y_+)_{1,0}
                        + (\bt\lm^{a}(x)+\bt\lm^{a}(y) - 2\lm^{a}(x)\lm^{a}(y))a \bigr).
        \end{align}
        The expression \eqref{eq-++al} is just a rearrangement of \eqref{eq star al}.
        For \eqref{eq-alal}, rearranging from \eqref{eq star 10},
        \begin{align*}
                2\al x_\al y_\al + (2\lm^{a}(x)\lm^{a}(y)-\bt(\lm^{a}(x)+\lm^{a}(y)))a
                        &        = (x*_ay)_{1,0} + 2\bt (x_+y_+)_{1,0} \\
                        & = (x*_ay)_+ - (x*_ay)_\al + 2\bt((x_+y_+)_+ - (x_+y_+)_\al).
        \end{align*}
        As $(x_+y_+)_+ = x_+y_+$, rearranging gives the final claim.

        \begin{lemma}
                \label{lem-c-c}
                We have that, if $\al-2\bt$ is invertible, then
                \begin{equation}
                        \begin{aligned}
                                \label{eq-ss}
                                ss = & \frac{1}{2}\frac{\al-\bt}{\al-2\bt}\Bigl[
                                \left(4(1-2\al)\lm_1^2
                                + 2(\al^2+\al\bt-4\bt)\lm_1
                                + \al\bt\lm_2
                                - \bt(\al^2+9\al\bt-4\bt^2-4\bt)
                                \right)a_0 \\
                                & + \left(
                                2(-\al^2+6\al\bt+\al-4\bt)\lm_1
                                - \bt(10\al\bt-4\bt^2+\al-6\bt)
                                \right)(a_1)_+
                                + \bt(\al-4\bt)(\al-\bt)(a_2)_+ \\
                                & + \frac{2}{\al-\bt}\left(
                                \al(3\al-2\bt-1)\lm_1
                                - \bt(6\al^2-10\al\bt+4\bt^2-\al)
                                \right)s
                                -\al\bt s_2
                                + \bt(\al-2\bt)s_2^f
                                \Bigr],
                        \end{aligned}
                \end{equation}
                from the equation, for any $x,y\in A$,
                \begin{equation}
                        \begin{aligned}
                                \label{eq (acx)(acy)}
                                (a\circ x)(a\circ y)
                                = {} & \left(\frac{\al}{2}-\bt\right)(x*_ay)_+
                                        - \frac{\al^2}{2(\al-2\bt)}(x*_ay)_\al
                                        + \bt(\al-\bt)x_+y_+ \\
                                & + \al^2\left(1+\frac{\bt}{\al-2\bt}\right)((\lm^{a}(y)-\bt)x_\al+(\lm^{a}(x)-\bt)y_\al) \\
                                & + (\lm^{a}(x)\lm^{a}(y)(1-\al) + (\lm^{a}(x)+\lm^{a}(y))\left(\frac{\al}{2}-1\right)\bt + \bt^2)a.
                        \end{aligned}
                \end{equation}
        \end{lemma}
        \begin{proof}
                Note that
                \begin{equation}
                        a\circ x + \bt x_+ = ax - \bt a - \bt x + \bt x_+ = ax -\bt x_\bt - \bt a = ax_+ - \bt a = (\lm^{a}(x)-\bt)a + \al x_\al.
                \end{equation}
                Now we can compute in two ways:
                \begin{align*}
                        (a\circ x + \bt x_+&)(a\circ y + \bt y_+)
                                = (a\circ x)(a\circ y) + \bt((a\circ x)y_+ + (a\circ y)x_+) + \bt^2 x_+y_+ \\
                                = {} & (a \circ x)(a\circ y) + \bt(x*_ay)_+ + \bt^2x_+y_+ \\
                        \intertext{since $a\circ x = (a\circ x)_+$,
                                so $(a\circ x)y_+ = ((a\circ x)y)_+$, on the one hand;
                                on the other,}
                                = {} & ((\lm^{a}(x)-\bt)a + \al x_\al)((\lm^{a}(y)-\bt)a + \al y_\al) \\
                                = {} & \al^2 x_\al y_\al + \al^2((\lm^{a}(x)-\bt)y_\al + (\lm^{a}(y)-\bt)x_\al) + (\lm^{a}(x)-\bt)(\lm^{a}(y)-\bt)a
                        \intertext{and using \eqref{eq-alal},}
                                = {} & \al\left(\frac{1}{2}(x*_ay)_{1,0} + \bt(x_+y_+)_{1,0}\right)
                                        + \al^2( (\lm^{a}(y)-\bt)x_\al + (\lm^{a}(x)-\bt)y_\al ) \\
                                &        + \left(\lm^{a}(x)\lm^{a}(y)(1-\al) +\bt\left(\frac{\al}{2}-1\right)(\lm^{a}(x)+\lm^{a}(y)) + \bt^2\right)a.
                \end{align*}
                So we may rearrange for our desired term:
                \begin{align*}
                        (a\circ x)(a\circ y)
                                = {} & \left(\frac{\al}{2}-\bt\right)(x*_ay)_+ - \frac{\al}{2}(x*_ay)_\al
                                        + \bt(\al-\bt)x_+y_+ - \al\bt(x_+y_+)_\al \\
                                & + \al^2( (\lm^{a}(y)-\bt)x_\al + (\lm^{a}(x)-\bt)y_\al ) \\
                                &        + (\lm^{a}(x)\lm^{a}(y)(1-\al) + (\lm^{a}(x)+\lm^{a}(y))\left(\frac{\al}{2}-1\right)\bt + \bt^2)a
                        \intertext{and finally, using \eqref{eq-++al},}
                                = {} & \left(\frac{\al}{2}-\bt\right)(x*_ay)_+ - \al\left(\frac{1}{2}+\frac{\bt}{\al-2\bt}\right)(x*_ay)_\al
                                        + \bt(\al-\bt)x_+y_+ \\
                                & + \al^2\left(1+\frac{\bt}{\al-2\bt}\right)(\lm^{a}(y)x_\al+\lm^{a}(x)y_\al)
                                        - \left(\frac{\al^2\bt^2}{\al-2\bt}+\al^2\bt\right)(x_\al+y_\al) \\
                                & + (\lm^{a}(x)\lm^{a}(y)(1-\al) + (\lm^{a}(x)+\lm^{a}(y))\left(\frac{\al}{2}-1\right)\bt + \bt^2)a,
                \end{align*}
                so we arrive at our conclusion \eqref{eq (acx)(acy)} after collecting terms.

                Now also note that,
                with respect to the $\Phi$-axis $a$, for any idempotent $e\in A$,
                \begin{equation}
                        \label{eq e+e+}
                        e_+e_+ = \frac{1}{2}e\circ e^{\tau(a)} + \left(\frac{1}{2}+\bt\right)e_+.
                \end{equation}
                By definition of $\tau(a)$, $e_+ = \frac{1}{2}(e+e^{\tau(a)})$.
                Hence, multiplying out,
                $e_+e_+ = \frac{1}{4}(e+e^{\tau(a)}) + \frac{1}{2}ee^{\tau(a)}$,
                and we rewrite the expression using the definition of $\circ$.
                For the product $ss$,
                we specialise: for any $x\in A$,
                \begin{equation}
                        \begin{aligned}
                                \label{eq cdot x cdot x}
                                (a\circ x)(a\circ x)
                                        = {} & \left(\frac{\al}{2}-\bt\right)(x*_ax)_+
                                                + \bt(\al-\bt)x_+x_+ \\
                                        & + (\lm^{a}(x)^2(1-\al)+\lm^{a}(x)(\al-2)\bt+\bt^2)a \\
                                        & + 2\al^2\left(1+\frac{\bt}{\al-2\bt}\right)(\lm^{a}(x)-\bt)x_\al
                                                - \frac{\al^2}{2(\al-2\bt)}(x*_ax)_\al.
                        \end{aligned}
                \end{equation}
                We need a number of auxiliary expressions.
                Observe that \eqref{eq-x-al} can be rewritten as,
                \begin{equation}
                        \label{eq x al alt}
                        x_\al = \frac{1}{\al}(a\circ x + \bt(a + x) -\lm^{a}(x)a - \bt x_\bt)
                                = \frac{1}{\al}(a\circ x + (\bt-\lm^{a}(x))a + \bt x_+).
                \end{equation}
                By application of \eqref{eq a(acx)} with $a_1$ in place of $a$,
                substituting $\frac{1}{2}(a_0+a_2)$ for ${a_0}^{a_1}_+$,
                \begin{align}
                        {a_1}*_{a_0}{a_1} = 2s{a_1}
                                = 2(\al-\bt)s + 2(\lm_1^f(1-\al)+\bt(\al-\bt-1)){a_1} + (\al-\bt)\bt({a_0}+a_2).
                \end{align}
                We immediately deduce,
                using $s_+ = s$ and $a_+ = a$,
                then using \eqref{eq cdot al}
                and \eqref{eq x al alt},
                \begin{align}
                        \label{eq b * b +}
                        & ({a_1}*_{a_0}{a_1})_+ = 2(\al-\bt)s + 2(\lm_1^f(1-\al)+\bt(\al-\bt-1))(a_1)_+ + (\al-\bt)\bt({a_0}+{(a_2)}_+), \\
                        & \begin{aligned}
                                \label{eq b * b al}
                                ({a_1}*_{a_0}{a_1})_\al = {} & 2(\al-\bt)s_\al + 2(\lm_1^f(1-\al)+\bt(\al-\bt-1))(a_1)_\al + (\al-\bt)\bt((a_0)_\al + {(a_2)}_\al) \\
                                = {} & \frac{2}{\al}((\al-\bt)^2 + \lm_1^f(1-\al) + \bt(\al-\bt-1))(s + (\bt-\lm_1){a_0} + \bt (a_1)_+) \\
                                &        + (\al-\bt)\frac{\bt}{\al}(s_2 + (\bt-\lm_2){a_0} + \bt{(a_2)}_+).
                        \end{aligned}
                \end{align}

                Finally, we substitute $a_1$ in place of $x$ in \eqref{eq cdot x cdot x},
                and substitute the expressions we collected
                for $({a_1}*_{a_0}{a_1})_+$ in \eqref{eq b * b +},
                for $(a_1)_+(a_1)_+$ in \eqref{eq e+e+},
                for $(a_1)_\al$ in \eqref{eq x al alt},
                and for $({a_1}*_{a_0}{a_1})_\al$ in \eqref{eq b * b al},
                to find
                \begin{gather}
                        \begin{aligned}
                                ss = {} &
                                        \left(\frac{\al}{2}-\bt\right)[2(\al-\bt)s + 2(\lm_1^f(1-\al)+\bt(\al-\bt-1))(a_1)_+ + (\al-\bt)\bt(a+{a^{\tau(b)}}_+)] \\
                                        & + \bt(\al-\bt)\left(\frac{1}{2}s_2^f + \left(\frac{1}{2}+\bt\right)(a_1)_+\right)
                                        + (\lm_1^2(1-\al)+\lm_1(\al-2)\bt+\bt^2)a_0 \\
                                        & + 2\al^2\left(1+\frac{\bt}{\al-2\bt}\right)\frac{\lm_1-\bt}{\al}(s + (\bt-\lm_1)a_0 + \bt (a_1)_+)
                                        - \frac{\al^2}{2(\al-2\bt)}\biggl[\frac{2}{\al}((\al-\bt)^2 + \lm_1^f(1-\al) \\
                                        & + \bt(\al-\bt-1))(s + (\bt-\lm_1)a_0 + \bt (a_1)_+)
                                        + (\al-\bt)\frac{\bt}{\al}(s_2 + (\bt-\lm_2)a_0 + \bt(a_2)_+)\biggr]
                        \end{aligned} \\
                        \begin{aligned}
                                = {} &
                                \left[ \left(\frac{\al}{2}-\bt\right)(\al-\bt)\bt
                                        + \lm_1^2(1-\al)+\lm_1(\al-2)\bt
                                        + \bt^2
                                        + 2\al\left(1+\frac{\bt}{\al-2\bt}\right)(\lm_1-\bt)(\bt-\lm_1)\right. \\
                                &\left.        - \frac{\al}{2(\al-2\bt)}\left(
                                                \bt(\al-\bt)(\bt-\lm^{a}(a^{\tau(b)}))
                                                + 2((\al-\bt)^2 + \lm_1^f(1-\al) + \bt(\al-\bt-1))(\bt-\lm_1)
                                                \right)
                                \right]a_0 \\
                                & + \left[ (\al-2\bt)(\lm_1^f(1-\al)+\bt(\al-\bt-1))
                                        + \bt(\al-\bt)\left(\frac{1}{2}+\bt\right)
                                        + 2\al\left(1+\frac{\bt}{\al-2\bt}\right)(\lm_1-\bt)\bt\right. \\
                                        &\left. - \frac{\al}{\al-2\bt}((\al-\bt)^2+\lm_1^f(1-\al)+\bt(\al-\bt-1
                                \right] (a_1)_+ \\
                                & + \left[
                                        \bt(\al-\bt)\left(\frac{\al}{2}-\bt\right)
                                        - \frac{\al^2}{2(\al-2\bt)}(\al-\bt)\frac{\bt^2}{\al}
                                \right](a_2)_+
                                        - \frac{\al\bt}{2(\al-2\bt)}(\al-\bt) s_2
                                        + \frac{\bt(\al-\bt)}{2} s_2^f \\
                                & + \left[ (\al-\bt)(\al-2\bt)
                                        + 2\al\left(1+\frac{\bt}{\al-2\bt}\right)(\lm_1-\bt)
                                        - \frac{\al}{\al-2\bt}((\al-\bt)^2 + \lm_1^f(1-\al) + \bt(\al-\bt-1))
                                \right] s \\
                        \end{aligned}
                \end{gather}
                which is the equation given in the statement,
                once its terms have been collected.
        \end{proof}

        \gap
        From now on, for the rest of this section,
        we will assume that $\al-2\bt$ is invertible.

        \begin{lemma}
                \label{lem-t-stab}
                If $\al-4\bt$ is invertible, then
                \begin{equation}
                        \begin{aligned}
                                \label{eq-a3}
                                a_3 & = a_{-2}
                                + \frac{1}{\bt(\al-\bt)(\al-4\bt)}\Biggl(
                                \left(4\al\bt\lm_1-2(\al^2-(4\bt+1)\al+4\bt)\lm_1^f
                                        -\al^2\bt-\al\bt(5\bt-1)+6\bt^2\right)a_{-1} \\
                                & + \frac{1}{\al-\bt}\Bigl( 4(3\al^2-(4\bt-1)\al+2\bt)\lm_1^2
                                        + (4\al(\al-1)\lm_1^f-6\al^3-2(3\bt-1)\al^2+2\al\bt(8\bt+1)-8\bt^2)\lm_1 \\
                                & - 4\al\bt(\bt-1)\lm_1^f - 2\al\bt(\al-\bt)\lm_2
                                        + 2\bt\al^3 + \bt(6\bt-1)\al^2-\al\bt^2(12\bt-1)+2\bt^3(2\bt-1)\Bigr)a_0 \\
                                & + \frac{1}{\al-\bt}\Bigl( 4\al(\al-1)\lm_1\lm_1^f
                                        - 4\al\bt(\bt-1)\lm_1
                                        + 4(3\al^2-4\al(4\bt-1)+2\bt){\lm_1^f}^2
                                        - 2(3\al^3+\al^2(3\bt-1) \\
                                &        - \al\bt(8\bt+1)+4\bt^2)\lm_1^f
                                        + 2\al\bt(\al-\bt)\lm_2 + 2\al^3\bt + \al^2\bt(6\bt-1)
                                        - \al\bt^2(12\bt+1) + 2\bt^3(2\bt+1) \Bigr)a_1 \\
                                & + \Bigl( 2(\al^2-\al(4\bt+1)+4\bt)\lm_1 - 4\al\bt\lm_1^f + \al^2\bt+\al\bt(5\bt+1)-6\bt^2 \Bigr)a_2 \\
                                & + \frac{1}{\al-\bt}\left(4\al(\al-2\bt+1)\lm_1+4\al(-\al+2\bt-1)\lm_1^f\right)s \Biggr) + \frac{4}{\al-4\bt}\left(s_2^f-s_2\right),
                        \end{aligned}
                \end{equation}
                and therefore the $R$-span of $B$ is stable under $T$.
                Thus we deduce expressions for
                \begin{equation*}
                        \label{eq terms via t}
                        a_2s,\quad
                        a_{-1}s,\quad
                        a_{-2}s,\quad
                        a_1s_2^f,\quad
                        a_{-1}s_2^f,\quad
                        a_2s_2,\quad
                        a_{-2}s_2,
                \end{equation*}
        \end{lemma}
        \begin{proof}
                Lemma \ref{lem-c-c} calculates $ss$
                using eigenspace decompositions and fusion rules with respect to $a_0$,
                and by repeating the computation
                with the roles of $a_0$ and $a_1$ swapped,
                we obtain an expression for $(ss)^f$ from \eqref{eq-ss}.
                On the other hand, $s$ is symmetric in $a_0$ and $a_1$
                so that $ss$ is invariant under interchange of $a_0$ and $a_1$.
                The equation \eqref{eq-a3} follows from the equality $(ss)^f = ss$.
                Since $a_3 = (a_{-2})^f$,
                we see the desired term in the expression for $(ss)^f$:
                \begin{equation}
                        \begin{aligned}
                                & \frac{1}{2}\frac{\al-\bt}{\al-2\bt}\left[
                                \left(4(1-2\al){\lm_1^f}^2
                                + 2(\al^2+\al\bt-4\bt)\lm_1^f
                                + \al\bt\lm_2^f
                                - \bt(\al^2+9\al\bt-4\bt^2-4\bt)
                                \right)b \right. \\
                                & + \left(
                                (-\al^2+6\al\bt+\al-4\bt)\lm_1^f
                                - \frac{1}{2}\bt(10\al\bt-4\bt^2+\al-6\bt)
                                \right)(a+a_2)
                                + \frac{1}{2}\bt(\al-4\bt)(\al-\bt)(a_{-1} + a_3) \\
                                & + \frac{2}{\al-\bt}\left(
                                \al(3\al-2\bt-1)\lm_1^f
                                - \bt(6\al^2-10\al\bt+4\bt^2-\al)
                                \right)s
                                -\al\bt s_2^f
                                \left. + \bt(\al-2\bt)s_2
                                \right].
                        \end{aligned}
                \end{equation}
                Rearranging yields the claim for $a_3$.

                Now observe that
                \begin{equation}
                        B^f = \{a_3,a_2,a_1,a_0,a_{-1},s,s_2^f,s_2\}.
                \end{equation}
                Out of these, the only term not already in $B$ is $a_3$.
                Our above expression for this term
                shows that $RB$ and $RB^f$ coincide.
                As $a_4 = a_{-3}^f$,
                that is, $a_3^{\tau(a_0)}$ with the roles of $a_0,a_1$ reversed,
                we have that $B^{\tau(a_1)}\subseteq RB^f = RB$ also,
                and since $B^{\tau(a_0)} = B$ this proves $RB^T = RB$.

                Now all of the terms in \eqref{eq terms via t}
                are in the $T,f$-orbit of $a_0s$ or $a_0s_2$:
                \begin{align*}
                        a_2s & = (a_0s)^{\tau(a_1)}, &
                        a_1s_2^f & = (a_0s_2)^f, \\
                        a_{-1}s & = (a_0s)^{\tau(a_1)f}, &
                        a_{-1}s_2^f & = (a_0s_2)^{f\tau(a_0)}, \\
                        a_{-2}s & = (a_0s)^{\tau(a_1)\tau(a_0)}, &
                        a_2s_2 & = (a_0s_2)^{\tau(a_1)}, \\
                        & & a_{-2}s_2 & = (a_0s_2)^{\tau(a_1)\tau(a_0)}.
                \end{align*}
        \end{proof}

        From now on we also assume that $\al-4\bt$ is invertible.
        \begin{lemma}
                \label{lem-a(bcbtaua)}
                We have that
                \begin{equation}
                        \begin{aligned}
                                a_0s_2^f
                                & = \frac{1}{\al-2\bt}\biggl(
                                        \frac{1}{\al-\bt}\Bigl(
                                        2(3\al-(4\bt-1)\al+2\bt)\lm_1^2 + 2\al(\al-1)\lm_1\lm_1^f
                                        + 2(-2\al^3+\al^2+\al\bt(2\bt-1))\lm_1 \\
                                        & + 4\bt(-\al^2+\al(\bt+1)-\bt)\lm_1^f
                                        - \al\bt(\al-\bt)\lm_2 + 4\al^3\bt - 2\bt(2\bt+1)\al^2+2\al\bt^2(\bt+1)-2\bt^4\Bigr)a_0 \\
                                        & - (2\al\lm_1+2(\al-1)\lm_1^f-2\al^2+(\bt+\frac{1}{2})\al-2\bt^2+\bt)\bigl(\bt(a_1+a_{-1})+s\bigr)
                                        + (\al-\bt)\bt^2\bigl(a_2+a_{-2}\bigr) \\
                                        & + 2\bt(\al-\bt)s_2\biggr).
                        \end{aligned}
                \end{equation}
                and from this follow expressions,
                using the $T$-invariance explained in Lemma \ref{lem-t-stab},
                for
                \begin{equation*}
                        a_2s_2^f,\quad
                        a_{-2}s_2^f,\quad
                        a_1s_2,\quad
                        a_{-1}s_2 \quad
                        \text{ and } \quad
                        a_{-2}a_1,\quad
                        a_{-2}a_2,\quad
                        a_{-1}a_2.
                \end{equation*}
        \end{lemma}
        \begin{proof}
                Recall from \eqref{eq x al alt} and \eqref{eq-x-bt} that
                \begin{equation}
                        \label{eq b0}
                        (a_1)_0 = \frac{1}{\al}\bigl(
                                (1-\al)\lm_1-\bt)a_0
                                + (\al-\bt)\frac{1}{2}(a_1+a_{-1})
                                - s\bigr).
                \end{equation}
                First we write down, using Lemma \ref{lem-t-stab},
                \begin{equation}
                        \begin{aligned}
                        (a_1)_0(a_1)_0 & = \frac{1}{2\al(\al-2\bt)}\biggl(
                                -\bigl(
                                        2\al(\al-2\bt+1)\lm_1^2
                                        + 2(\al-1)\lm_1\lm_1^f
                                        + 2(-\al^2-2\al\bt+2\bt^2+\bt)\lm_1 \\
                                &        - 2\bt(\al-1)\lm_1^f
                                        - \bt(\al-\bt)\lm_2
                                        + 3\bt\al^2 - 3\bt^3-2\bt^2\bigr)a_0 \\
                                & + (\al-\bt)\bigl(4\bt\lm_1 + 2(\al-1)\lm_1^f - (2\bt-1)\al - 4\bt^2\bigr)(a_1)_+
                                        + \bt(\al-\bt)(a_2)_+ \\
                                & (2(\al-\bt)\lm_1+(\al-1)\lm_1^f-2\al^2+2\al\bt+\bt)s
                                - \bt(\al-\bt)s_2 + (\al-\bt)(\al-2\bt)s_2^f
                        \biggr).
                        \end{aligned}
                \end{equation}
                From the fusion rule $0\star0=0$
                we deduce that $a_0((a_1)_0(a_1)_0)=0$.
                The only product we do not already know is $a_0s_2^f$,
                so substituting and rearranging gives the result.

                The second set of equations
                follows from the fact that $T$ is transitive on $a_{-2},a_{-1},a_0,a_1,a_2$,
                and we have expressions for $a_0x$ for all $x\in B$,
                so we can take any product $a_ix$
                and find a representative $a_0x'$ in the $T$-orbit of $a_ix$
                with an expression for this product in $RB$.
                As $RB$ is $T$-closed, this allows us to calculate any $a_ix$.
        \end{proof}

        \begin{lemma}
                \label{lem-acbacataub}
                We can find an expression for $ss_2$ in $RB$,
                and hence get an equation for $ss_2^f$ too.
                We also have expressions for $s_2s_2$
                and $s_2^fs_2^f$ in $RB$.
        \end{lemma}
        \begin{proof}
                We will derive the first equality starting from the equation
                \begin{equation}
                        a_0((a_1)_0(a_2)_\al + (a_1)_0(a_2)_0) = \al (a_1)_0(a_2)_\al
                \end{equation}
                which follows from the fusion rules.
                The key is that the contributions of $ss_2$
                from each of the terms $(a_1)_0(a_2)_\al$ and $(a_1)_0(a_2)_0$ cancel
                on the lefthand side,
                but not on the righthand side.

                We know $(a_1)_0$ from \eqref{eq b0}.
                In the same way we calculate
                \begin{equation}
                        (a_2)_0 = \frac{1}{\al}\bigl(
                                (1-\al)\lm_2-\bt)a_0
                                + (\al-\bt)\frac{1}{2}(a_2+a_{-2})
                                - s_2\bigr),
                \end{equation}
                and also get, from \eqref{eq x al alt},
                \begin{equation}
                        (a_2)_\al = \frac{1}{\al}\bigl(
                                (\bt-\lm_2)a_0
                                + \frac{1}{2}\bt(a_2+a_{-2})
                                + s_2\bigr).
                \end{equation}
                The previous lemmas are enough to calculate the products,
                so arriving at the answer is a matter of rearranging the copious terms.

                Then $ss_2^f=(ss_2)^f$.
                The third and forth products promised follow from
                \begin{equation}
                        a_0((a_2)_0(a_2)_\al + (a_2)_0(a_2)_0) = \al (a_2)_0(a_2)_\al
                \end{equation}
                using the same method.
        \end{proof}

        \begin{lemma}
                \label{lem-ult-calc}
                We can express $s_2s_2^f$ in $RB$.
        \end{lemma}
        \begin{proof}
                Recall that idempotents are preserved by automorphisms,
                so that $a_3a_3=a_3$.
                The expression afforded in Lemma \ref{lem-t-stab},
                and knowing all other products,
                allows us to express $s_2s_2^f$.
        \end{proof}

        Altogether the previous sequence of Lemmas proves
        \begin{theorem}
                \label{thm-mt}
                Suppose that $R$ is an associative $R''[(\al-2\bt)^{-1},(\al-4\bt)^{-1}]$-algebra.
                If $A$ is an algebra over $R$
                generated by two $\Phi(\al,\bt)$-axes $a_0,a_1$,
                then $A$ is spanned by
                \begin{gather*}
                        B = \{ a_{-2},a_{-1},a_0,a_1,a_2,s,s_2,s_2^f \}.
                \end{gather*}
                with the multiplication table described by the previous Lemmas \ref{lem-c-c} to \ref{lem-ult-calc}.
                \qed
        \end{theorem}
        In particular, when $U$ is the universal $2$ generated $\Phi$-axial algebra
        over ring $R_U$ from Theorem \ref{thm-universe}
        with the extra condition that $\al-2\bt,\al-4\bt$ be invertible,
        then $R_U$ is a quotient of the polynomial ring
        \begin{equation}
                \ZZ[\sfrac{1}{2},\al,\bt,\al^{-1},\bt^{-1},(\al-\bt)^{-1},(\al-2\bt)^{-1},(\al-4\bt)^{-1}]
        \end{equation}
        and $U$ has the multiplication table of Theorem \ref{thm-mt} over $R_U$.
        (That $U$ cannot be spanned by a set of size $7$
        is clear by the existence of $\sh{6A}$,
        a quotient algebra which is $8$-dimensional over $\QQ$.)

        \gap
        We now define a form $(,)$ on the algebra of Theorem \ref{thm-mt}
        by setting $(a_i,a_i) = 1$ and applying several relations
        which are necessary for $(,)$ to be Frobenius.
        (We do not apply all such relations to guarantee that $(,)$ is Frobenius,
        since their computational complexity exceeds our capabilities.
        But applying all these relations is a valid possibility,
        which was exploited especially for Majorana alebras \cite{ipss}.)
        Firstly, $\lm^{a_i}(a_j) = \frac{(a_i,a_j)}{(a_i,a_i)} = (a_i,a_j)$.
        Secondly,
        \begin{equation}
                (a_i,s) = (a_i,a_ia_{i+1} - \bt(a_i+a_{i+1}))
                        = (a_i,a_{i+1}) - \bt(a_i,a_i) - \bt(a_i,a_{i+1})
                        = \lm_1(1-\bt) - \bt.
        \end{equation}
        We likewise compute the remaining values of the form
        up to action by $T$ and $f$.
        \begin{align}
                        (a_i,s_2) & = \lm_2(1-\bt) - \bt, \text{ for } i= -2,0,2 \\
                        (a_i,s_2) & = (\bt\lm_2 - (2\bt - 1)\lm - \bt), \text{ for } i = -1,1 \\
                        (s,s) & = \frac{1}{2}(\bt(\al - \bt)\lm_2 - 2(\al - 1)\lm^{2} + 2(\al + 2\bt^{2} - 4\bt)\lm - \bt\al + 5\bt^{2}), \\
                        \begin{gathered}(s,s_2) \\ $ $ \end{gathered} & \, \begin{aligned}
                                = {} & \frac{1}{2}\frac{(\al - 2\bt)(\al - \bt)}{(\al - 4\bt)(\al - 2\bt)(\al - \bt)}\bigl(\bt(-6\al\lm + 4\al^{2} - (2\bt + 3)\al + 4\bt^{2} + 6\bt)\lm_2 + 8(2\al - 1)\lm^{3} \\
                                & - 4(\al + 2)(2\al - 1)\lm^{2} + 2(4\al^{2} + (2\bt^{2} - 1)\al - 8\bt^{3} + 12\bt^{2} - 4\bt)\lm - 4\bt\al^{2} + \bt(6\bt + 1)\al - 20\bt^{3} + 2\bt^{2}\bigr),\end{aligned} \\
                        \begin{gathered}(s_2,s_2)
                                \\\\\\\\\\\\\\ $ $
                        \end{gathered} & \, \begin{aligned}
                                = {} & \frac{1}{2}\frac{1}{\bt(\al-4\bt)^2(\al-\bt)}\bigl(-2\bt(\al - 4\bt)(\al - \bt)(\al^{2} - (3\bt + 1)\al + 4\bt)\lm_2^{2} + (4(\al^{4} - (9\bt + 2)\al^{3} \\
                                & + (22\bt^{2} + 19\bt + 1)\al^{2} + 2\bt(8\bt^{2} - 21\bt - 4)\al - 8\bt^{3} + 16\bt^{2})\lm^{2} + 2\bt(2(2\bt - 1)\al^{3} - \bt(74\bt + 11)\al^{2} \\
                                & + 2\bt(2\bt^{2} + 47\bt + 4)\al - 32\bt^{2})\lm + 2\bt(2\bt + 1)\al^{4} + 2\bt^{2}(6\bt - 11)\al^{3} + \bt^{2}(12\bt^{2} + 88\bt - 1)\al^{2} \\
                                & + 4\bt^{3}(14\bt^{2} - 61\bt + 1)\al - 48\bt^{6} + 2^{7}\bt^{5} + 12\bt^{4})\lm_2 + 16(2\al - 1)(\al^{2} - (6\bt + 1)\al + 4\bt)\lm^{4} \\
                                & - 8(2\al - 1)(\al^{3} - (6\bt - 1)\al^{2} - (2\bt + 1)(5\bt + 2)\al + 4\bt^{3} + 6\bt^{2} + 8\bt)\lm^{3} - 4(-3\al^{4} \\
                                & + (20\bt^{2} + 17\bt + 4)\al^{3} - (8\bt^{3} - 40\bt^{2} + 6\bt + 1)\al^{2} + 2\bt(2\bt^{2} - 40\bt - 1)\al + 28\bt^{2})\lm^{2} \\
                                & + 2\bt(6(6\bt + 1)\al^{3} + (58\bt^{2} - 33\bt - 2)\al^{2} - 2\bt(2\bt^{2} + 43\bt - 2)\al + 32\bt^{2})\lm - 4\bt^{2}\al^{4} \\
                                & - 2\bt^{2}(2\bt + 1)\al^{3} - \bt^{2}(84\bt^{2} - 22\bt - 1)\al^{2} + 4\bt^{3}(34\bt^{2} + 7\bt - 1)\al - 80\bt^{6} - 12\bt^{4}\bigr),\end{aligned} \\
                        \begin{gathered}(s_2,s_2^f)
                                \\\\\\\\\\\\ $ $
                        \end{gathered} & \, \begin{aligned}
                                = {} & \frac{1}{2}\frac{1}{(\al-4\bt)^2(\al-\bt)}
                                \bigl(-2\bt\al(\al - 4\bt)(\al - \bt)\lm_2^{2} + (8(\al^{3} - 3\bt\al^{2} + \bt(8\bt + 1)\al - 4\bt^{2})\lm^{2} \\
                                & - 2((26\bt + 3)\al^{3} - (70\bt^{2} + 27\bt + 1)\al^{2} + 2\bt(46\bt^{2} + 27\bt + 3)\al - 48\bt^{3} - 8\bt^{2})\lm + 16\bt\al^{4}\\
                                & - 2\bt(22\bt + 3)\al^{3} + \bt(72\bt^{2} - 2\bt + 1)\al^{2} - 4\bt^{3}(22\bt - 7)\al + 80\bt^{5} - 32\bt^{4} - 4\bt^{3})\lm_2\\
                                & - 32\al(2\al - 1)\lm^{4} + 8(2\al - 1)(6\al^{2} - (10\bt - 3)\al + 12\bt^{2} + 2\bt)\lm^{3} - 4(8\al^{4} - 20(\bt - 1)\al^{3}\\
                                & + (24\bt^{2} - 32\bt - 15)\al^{2} + 2(26\bt^{2} + 14\bt + 1)\al - 32\bt^{2} - 4\bt)\lm^{2} + 2(16\al^{4} + 2(2\bt^{2} - 9\bt - 4)\al^{3} \\
                                & - (36\bt^{3} - 14\bt^{2} + 8\bt - 1)\al^{2} + 2\bt(48\bt^{3} - 2\bt^{2} + 27\bt + 1)\al - 64\bt^{5} + 64\bt^{4} - 64\bt^{3} - 8\bt^{2})\lm \\
                                & - 16\bt\al^{4} + 4\bt(11\bt + 2)\al^{3} - \bt(72\bt^{2} + 8\bt + 1)\al^{2} + 4\bt^{3}(22\bt - 5)\al - 80\bt^{5} + 32\bt^{4} + 4\bt^{3}\bigr).
                        \end{aligned}
        \end{align}
        By this definition,
        $(,)$ does not make the algebra Frobenius;
        for example, $(a_{-1},a_2) \neq (a_{-1},a_{-1}a_2)$.
        However, some quotients of $(,)$
        will turn out to be Frobenius in the sequel.

\section{Fusion rules and covers}
        \label{sec-fus-rule}
        
        Not all fusion rules have been enforced yet.
        We describe, given a $\Phi(\al,\bt)$-dihedral algebra,
        how to find its generalisation, called an axial cover,
        by finding smaller ideals, coming from the fusion rules,
        in the universal algebra previously described.
        We also introduce the extra assumption that the coefficient functions $\lm^e,\lm^f$
        are symmetric.

        \gap
        Suppose that $R$ is a ring satisfying the assumption in Theorem~\ref{thm-mt},
        so that $R$ is an associative algebra over
        \begin{equation}
                \tag{\ref{eq r0}}
                R_0 = \ZZ[\al,\bt,\al^{-1},\bt^{-1},(\al-\bt)^{-1},(\al-2\bt)^{-1},(\al-4\bt)^{-1}].
        \end{equation}
        Let $A_R$ be the free $R$-module on $B = \{a_{-2},a_{-1},a_0,a_1,a_2,s,s_2,s_2^f\}$
        together with the multiplication from Theorem~\ref{thm-mt}.
        Then $a_0,a_1\in A_R$ are not necessarily $\Phi(\al,\bt)$-axes,
        since their eigenvectors do not satisfy the fusion rules in general,
        as we will see in Lemma~\ref{lem-mother-rels}.
        Therefore $A_R$ is not necessarily a $\Phi(\al,\bt)$-dihedral algebra.

        However, Theorem~\ref{thm-mt} shows that any $\Phi(\al,\bt)$-dihedral algebra
        over a ring $R$
        satisfies the multiplication rules given in Section~\ref{sec-mt},
        and therefore is a quotient of $A_R$.
        In particular, Theorem~\ref{thm-universe} asserts that there exists a ring $R_U$,
        which is a quotient of $R_0$ by some ideal $J^{R_0}_{\Phi(\al,\bt)}$,
        and an algebra $U$ over $R_U$ which is the universal $\Phi(\al,\bt)$-dihedral algebra.
        Hence $U$ is a quotient of $A_{R_U}$ by some ideal $I^{R_0}_{\Phi(\al,\bt)}$.
        (We use the subscript $\Phi(\al,\bt)$ in our notation for these ideals
        to indicate that they come solely from the fusion rules.)
        While actually finding $U$ over $R_U$ is beyond our reach,
        we work with $A_{R_0}$ over $R_0$
        as an approximation to the universal object.

        Short of classifying all the $\Phi(\al,\bt)$-dihedral algebras,
        we use our results to significantly generalise the known $\Phi(\al,\bt)$-dihedral algebras
        by a pullback of ideals, as we now explain.

        Namely, suppose that $\sh{nX}_R$ is a $\Phi(\al,\bt)$-dihedral everywhere-faithful $R$-algebra,
        where $R$ is an associative algebra over $R_0$.
        Then, by Theorem~\ref{thm-universe},
        there exist matching ideals
        \begin{equation}
                J^R_{\sh{nX}} \subseteq R_U\tensor_\ZZ R, \quad
                I^R_{\sh{nX}} \subseteq U_R = (R_U\tensor_\ZZ R)U
        \end{equation}
        such that
        \begin{equation}
                R \cong (R_U\tensor_\ZZ R)/J^R_{\sh{nX}}, \quad
                \sh{nX} \cong U_R/I^R_{\sh{nX}}.
        \end{equation}
        Recall that, for the two ideals $J^R_{\sh{nX}},I^R_{\sh{nX}}$ to match,
        we must have $J^R_{\sh{nX}}U_R\subseteq I^R_{\sh{nX}}$,
        and $\lm^{a_i}(x) \in J^R_{\sh{nX}}$ for $i=0,1$ and any $x\in I^R_{\sh{nX}}$.
        Note that, if $R$ is a domain, then $J^R_{\sh{nX}}$ is a prime ideal.

        Suppose we have matching ideals
        \begin{equation}
                \begin{gathered}
                        \label{eq weak cov ideals}
                        {J'}^R_{\sh{nX}} \subseteq J^R_{\sh{nX}} \subseteq R_U\tensor_\ZZ R, \quad
                        {I'}^R_{\sh{nX}} \subseteq I^R_{\sh{nX}} \subseteq U_R = (R_U\tensor_\ZZ R)U
                \end{gathered}
        \end{equation}
        such that $\sh{nX'} = U_{(R_U\tensor_\ZZ R)/{J'}^R_{\sh{nX}}}/{I'}^R_{\sh{nX}}$
        is a $\Phi(\al,\bt)$-dihedral everywhere faithful $(R_U\tensor_\ZZ R)/{J'}^R_{\sh{nX}}$-algebra.
        We call $\sh{nX'}$ a {\em weak (axial) cover} of $\sh{nX}$.
        Note $\sh{nX}$ is a quotient of $\sh{nX'}$.
        \begin{definition}
                \label{def axcov}
                Let $\hat J^R_{\sh{nX}} = \bigcap {J'}^R_{\sh{nX}}$,
                $\hat I^R_{\sh{nX}} = \bigcap {I'}^R_{\sh{nX}}$ over all matching ideals
                ${J'}^R_{\sh{nX}},{I'}^R_{\sh{nX}}$ as above.
                The $(R_U\tensor_\ZZ R)/\hat J^R_{\sh{nX}}$-algebra $U_R/\hat I^R_{\sh{nX}}$
                is the {\em (axial) cover} of $\sh{nX}$.
                \index{axial algebra!cover}
        \end{definition}
        Note that it is possible that there are infinite descending chains of such ideals ${J'}^R_{\sh{nX}},{I'}^R_{\sh{nX}}$,
        so it is not {\em a priori} clear that their intersection, or the axial cover,
        is well-defined.

        If the ideals $\hat J^R_{\sh{nX}},\hat I^R_{\sh{nX}}$ are strictly smaller than $J^R_{\sh{nX}},I^R_{\sh{nX}}$,
        this means that $\sh{nX}$ is subject to additional constraints
        other than those coming from the fusion rules.
        Then $\sh{nX}$ is a proper quotient of its axial cover,
        which is its largest generalisation as an axial algebra.
        Our specific application will be to the Norton-Sakuma algebras $\sh{nX}$,
        listed in Table~\ref{tbl-NS}.

        Since $R_U$ and therefore $U_R$ are not available to work with,
        we will use our approximation $A_{R_0}$ as follows.
        We still consider a fixed $\Phi(\al,\bt)$-dihedral algebra $\sh{nX}$
        over an everywhere faithful ring $R$ which is an associative $R_0$-algebra.
        Then
        \begin{gather}
                A_{R_0\tensor_\ZZ R}
                \xrightarrow{I^R_{\Phi(\al,\bt)}} U_R
                \xrightarrow{I^R_{\sh{nX}}} \sh{nX}, \\
                {R_0\tensor_\ZZ R}
                \xrightarrow{J^R_{\Phi(\al,\bt)}} (R_U\tensor_\ZZ R)/J^R_{\Phi(\al,\bt)}
                \xrightarrow{J^R_{\sh{nX}}} R
        \end{gather}
        shows, in the top line, the relation among the algebras,
        and in the bottom line the relation among their rings.
        Instead of finding
        \begin{equation}
                \hat I^R_{\sh{nX}}\subseteq U_R \text{ and }
                \hat J^R_{\sh{nX}}\subseteq(R_U\tensor_\ZZ R),
        \end{equation}
        we will try to find ideals
        \begin{equation}
                \bar I^R_{\sh{nX}}\subseteq A_R \text{ and }
                \bar J^R_{\sh{nX}}\subseteq(R_0\tensor_\ZZ R)
        \end{equation}
        such that
        \begin{equation}
                \bar I^R_{\sh{nX}}/I^R_{\Phi(\al,\bt)} = \hat I^R_{\sh{nX}} \text{ and }
                \bar J^R_{\sh{nX}}/J^R_{\Phi(\al,\bt)} = \hat J^R_{\sh{nX}}.
        \end{equation}
        Notice that $A_R/\bar I^R_{\sh{nX}}$ as an $(R_0\tensor_\ZZ R)/\bar J^R_{\sh{nX}}$-algebra
        is exactly the axial cover of $\sh{nX}$.

        As it turns out, a subset of the fusion rules
        will be sufficient to generate $\bar I^R_{\sh{nX}}$ and $\bar J^R_{\sh{nX}}$,
        which will considerably shorten our work.
        In practice, we can calculate $\bar J^R_{\sh{nX}}$ as follows.
        As the cover is everywhere faithful over its ring $(R_0\tensor_\ZZ R)/\bar J^R_{\sh{nX}}$,
        the ring is a domain and $\bar J^R_{\sh{nX}}$ must be a prime ideal.
        Thus we have
        \begin{lemma}
                \label{lem cover smallest factor}
                For any $p\in J^R_{\Phi(\al,\bt)}$,
                let $p_{\sh{nX}}$ be the smallest factor of $p$ contained in $\bar J^R_{\sh{nX}}$.
                Then $\bar J^R_{\sh{nX}} = (p_{\sh{nX}}\mid p\in J^R_{\Phi(\al,\bt)})$.
                \qed
        \end{lemma}

        Note that if $\sh{nX}$ is a $\Phi(\bar\al,\bar\bt)$-axial algebra
        then $\al-\bar\al,\bt-\bar\bt\in\hat J^R_{\sh{nX}}$.
        Therefore if $p(\al)$ is an irreducible polynomial in $\bar J^R_{\sh{nX}}\subseteq\hat J^R_{\sh{nX}}$
        and coprime to $\al-\bar\al$
        then the ideal $(p(\al)) + (\al-\bar\al)$
        is equal to $(1)$, the entire ring,
        which implies $\hat J^R_{\sh{nX}} = (1)$
        and ${\sh{nX}}$ is the trivial algebra.
        The same argument applies with respect to $\bt$.
        This is a useful restriction on relations inside $\hat J^R_{\sh{nX}}$.

        \gap
        We now provide the eigenvectors and fusion rules
        used to find the ideals $\bar I^R_{\sh{nX}},\bar J^R_{\sh{nX}}$.
        \begin{lemma}
                \label{lem-eigvect-ising}
                In the algebra $A_R$ (coming from Theorem~\ref{thm-mt}),
                \begin{align}
                        \begin{aligned}
                                & A^{a_0}_1 = \la {a_0}\ra. \\
                                & \!\begin{aligned} A^{a_0}_0 & \text{ contains } 
                                        z = ((1-\al)\lm_1-\bt)a_0 + \frac{1}{2}(\al-\bt)(a_1+a_{-1}) - s, \\
                                        & \quad \begin{aligned}
& zz = -\frac{1}{4}\frac{1}{\al-2\bt}\Bigl(
 \bt\al(\al-\bt)^{2} (a_{-2} + a_2) \\
& \quad\quad + \al(\al-\bt)(2\lm_1^f - \al - 4\bt\lm_1 + 4\bt^{2} - 2\al\lm_1^f + 2\al\bt) (a_{-1} + a_1) \\
& \quad\quad + 2\al(-2\lm_1\lm_1^f + 2\bt\lm_1^f + 2\bt\lm_1 - 2\bt^{2} + \bt^{2}\lm_2 + 4\bt^{2}\lm_1 - 3\bt^{3} + 2\al\lm_1\lm_1^f + 2\al\lm_1^{2} \\
& \quad\quad\quad\quad\quad - \al\bt\lm_2 - 2\al\bt\lm_1^f - 4\al\bt\lm_1 - 2\al^{2}\lm_1 +3\al^{2}\bt - 4\al\bt\lm_1^{2} + 2\al^{2}\lm_1^{2}) a_0 \\
& \quad\quad + 4\al(\lm_1^f - \bt + 2\bt\lm_1 - \al\lm_1^f - 2\al\lm_1 - 2\al\bt + 2\al^{2}) s \\
& \quad\quad + 2\bt\al(\al-\bt) s_2 - 2\al(\al-\bt)(\al-2\bt) s_2^f
                                        \Bigr), \text{ and}\end{aligned} \\
                                        & \quad z_2 = ((1-\al)\lm_2-\bt)a + \frac{1}{2}(\al-\bt)(a_2+a_{-2}) - s_2
                                . \end{aligned} \\
                                & \!\begin{aligned} A^a_\al \text{ contains } 
                                        & x = (\bt-\lm_1)a_0 + \frac{1}{2}\bt(a_1 + a_{-1}) + s, \quad 
                                         x_2 = (\bt-\lm_2)a_0 + \frac{1}{2}\bt(a_2 + a_{-2}) + s_2
                                . \end{aligned} \\
                                & A^a_\bt \text{ contains } 
                                        y = a_1 - a_{-1},\quad
                                        y_2 = a_2 - a_{-2}
                                .
                        \raisetag{\baselineskip}
                        \end{aligned}
                \end{align}
        \end{lemma}
        \begin{proof}
                Since we have the multiplication table,
                it is a routine calculation to check that the vectors listed
                are eigenvectors of the appropriate eigenvalues.
                That $a_0$ spans $A^{a_0}_1$ comes from the assumption of primitivity.
        \end{proof}
        Similarly, eigenvectors of $a_1$ can be easily calculated
        by interchanging the r\^oles of $a_0$ and $a_1$ in the above.

        \begin{lemma}
                \label{lem-mother-rels}
                The coefficient of $a_{-1}$ in $a_0(z_2z_2)$ is
                \begin{equation}
                        \label{eq-mother}
                        \begin{aligned}
                        \frac{-2}{\bt(\al-4\bt)^3} &
                        \cdot ((\al-1)\lm_1+\bt(2\bt-2\al+1)) \\
                        & \cdot (2(\al(\al-1)-6\al\bt+4\bt)\lm_1+(10\bt+1)\al\bt-2\bt^2(2+3\bt)) \\
                        & \cdot (4(2\al-1)\lm_1^2-2(\al^2+6\bt\al-4\bt)\lm_1-\al\bt\lm_2+2\al^2\bt+4(\al-1)\bt^2).
                        \end{aligned}
                \end{equation}
                The other nonzero coefficients are those of $a_{-2},a_0,a_1,a_2$.
        \end{lemma}
        \begin{proof}
                We establish the formula by direct computation using the results of Section~\ref{sec-mt}.
                That the coefficient is unique
                follows by our assumption that we are working over an everywhere faithful ring.
        \end{proof}

        \gap
        Finally, to simplify our calculations,
        we introduce an assumption on the coefficients.

        {\em We will from now on}, for simplicity of calculation,
        assume that $\lm_1 = \lm_1^f$ and $\lm_2 = \lm_2^f$.
        This assumption is realised in at least three different situations:

        \begin{lemma}
                \label{lem-odd-quots}
                Suppose that $a_0 = a_{2n+1}$ for some $n$.
                Then $\lm^{a_i}(a_j) = \lm^{a_j}(a_i)$ for all $i,j$,
                                so that $\lm_i = \lm_i^f$,
                and $s_2 = s_2^f$.
        \end{lemma}
        \begin{proof}
                The group $T = \la\tau(a_0),\tau(a_1)\ra$
                acts transitively on pairs $\{a_i,a_j\}$ with $i-j\equiv_20$
                and on pairs $\{a_i,a_j\}$ with $i-j\equiv_21$.
                If $a_0 = a_{2n+1}$,
                then $\tau(a_{n+1})\in T$ swaps $a_0 = a$ and $a_1 = b$,
                and $\{a_0,a_2\}$ is swapped with $\{a_1,a_{-1}\}$.
                As $\tau(a_{n+1})$ also swaps the respective $1$-eigenspaces,
                $\lm^{a_i}(a_j) = \lm^{a_i}(a_j)^{\tau(a_{n+1})} = \lm^{a_j}(a_i)$.
                Also $\tau(a_{n+1})$ interchanges $s_2,s_2^f$,
                but $T$ is generated by $\tau(a_0),\tau(a_1)$ which both fix $s_2$,
                so $\tau(a_{n+1})$ must fix $s_2$ and therefore $s_2 = s_2^f$.
        \end{proof}

        Secondly, if the algebra is Frobenius,
        then $\lm^a(b) = \frac{(a,b)}{(a,a)}a$ for any primitive nonsingular $a$;
        in particular, if $(a_0,a_0) = (a_1,a_1)$
        then $\lm_1 = \lm_1^f$ and furthermore,
        by the same argument as Lemma 7.4 in \cite{hrs-sak},
        $\lm_2 = \lm_2^f$.
        (In fact, it can be shown that $\lm^{a_i}(a_j)$ depends only on $\size{i-j}$.)

        Finally, it can be shown that, if $a$ is a $\Phi(\al)$-axis
        and $b$ is a $\Phi(\bt)$-axis with $ab\neq 0$, then $\bt=1-\al$,
        $b' = \id-b$ is a $\Phi(\al)$-axis,
        $a$ and $b'$ generate the algebra
        and $\lm^a(b') = \lm^{b'}(a)$.
        In other words,
        the property $\lm^a(b) = \lm^b(a)$ follows from the axioms
        in $\Phi(\al)$-dihedral algebras if one allows the swap $b'$ for $b$.

        We are not sure if similarly $\lm^a(b) = \lm^b(a)$
        is a consequence of the axioms for Ising-axial algebras,
        but no counterexamples are known.

        In the following sections,
        we find the covers of the Norton-Sakuma algebras $\sh{nX}$ over $\QQ$,
        described below,
        for $\sh{nX}$ one of $\sh{4A},\sh{4B},\sh{5A},\sh{6A}$.
        The cover for $\sh{3A}$ is more complicated to determine,
        so we instead find a {\em weak cover}:
        we calculate ideals $\bar I_{\sh{3A}}\subseteq\bar I'_{\sh{3A}}\subseteq\hat I_{\sh{3A}}$
        and        $\bar J_{\sh{3A}}\subseteq\bar J'_{\sh{3A}}\subseteq\hat J_{\sh{3A}}$
        which conjecturally coincide with $\bar I_{\sh{3A}}$ and $\bar J_{\sh{3A}}$.

        The Norton-Sakuma algebras $\sh{nX}$
        are given by Table \ref{tbl-NS}
        together with the formulas
        \begin{gather}
                a_i = a_{i \bmod n},\quad
                a_ia_{i+1} = \frac{1}{64}(a_i+a_{i+1})+s, \quad
                a_ia_{i+2} = \frac{1}{64}(a_i+a_{i+2})
                        + \begin{cases}
                        s_2 & \text{ if } i\equiv_2 0, \\
                        s_2^f & \text{ if } i\equiv_2 1.
                \end{cases}
        \end{gather}
        Under the isomorphism type $\sh{nX}$,
        we give a spanning set and notes;
        the other column contains all the products
        necessary to calculate in the algebra.
        
                \begin{table}
                \begin{center}
        \begin{tabular}{ll}
                \hline
                Description & Products \\
                \hline
                \specialcell[1.2]{l}{
                        $\sh{1A}$ \\
                        $a_0$
                } & \specialcell[1.2]{l}{
                        $a_0a_0 = a_0$ \\
                        $(a_0,a_0) = 1$
                }\\\hline
                \specialcell[1.2]{l}{
                        $\sh{2B}$ \\
                        $a_0,a_1$
                } & \specialcell[1.2]{l}{
                        $a_0a_1 = 0$ \\
                        $(a_0,a_1) = 0$
                }\\\hline
                \specialcell[1.2]{l}{
                        $\sh{2A}$ \\
                        $a_0,a_1,a_2$
                } & \specialcell[1.2]{l}{
                        $a_ia_j = \frac{1}{8}(a_i+a_j-a_k)$ \\
                        $(a_i,a_j) = \frac{1}{8}$
                }\\\hline
                \specialcell[1.2]{l}{
                        $\sh{3C}$ \\
                        $a_0,a_1,a_2$
                } & \specialcell[1.2]{l}{
                        $a_ia_j = \frac{1}{64}(a_i+a_j-a_k)$ \\
                        $(a_i,a_j) = \frac{1}{64}$
                }\\\hline
                \specialcell[1.2]{l}{
                        $\sh{3A}$ \\
                        $a_0,a_1,a_2,s$ \\
                        $ $
                } & \specialcell[1.2]{l}{
                        $a_is = \frac{7}{2^{11}}(4a_i+a_j+a_k) + \frac{7}{32}s$ \\
                        $ss = \frac{147}{2^{16}}(a_i+a_j+a_k) - \frac{63}{2^{11}}s$ \\
                        $(a_i,a_j) = \frac{13}{256}$
                }\\\hline
                \specialcell[1.2]{l}{
                        $\sh{4A}$ \\
                        $a_{-1},a_0,a_1,a_2,s$ \\
                        $\la a_i,a_{i+2}\ra\cong\sh{2B}$
                } & \specialcell[1.2]{l}{
                        $a_is = \frac{1}{2^{11}}(7(a_{i+1}+a_{i-1})-2a_i)+\frac{7}{32}s$ \\
                        $ss = \frac{21}{2^{15}}(a_{-1}+a_0+a_1+a_2)-\frac{11}{2^9}s$ \\
                        $(a_i,a_{i+1}) = \frac{1}{32}$
                }\\\hline
                \specialcell[1.2]{l}{
                        $\sh{4B}$ \\
                        $a_{-1},a_0,a_1,a_2,s$ \\
                        $\la a_i,a_{i+2}\ra\cong\sh{2A}$ \\
                        $ $
                } & \specialcell[1.2]{l}{
                        $a_ia_{i+2} = -\frac{1}{8}(a_{i+1}+a_{i-1})-8s$ \\
                        $a_is = \frac{1}{2^{11}}(7(a_{i+1}+a_{i-1})-26a_i)+\frac{7}{32}s$ \\
                        $ss = \frac{7}{2^{15}}(a_{-1}+a_0+a_1+a_2)-\frac{3}{2^9}s$ \\
                        $(a_i,a_{i+1}) = \frac{1}{64}$
                }\\\hline
                \specialcell[1.2]{l}{
                        $\sh{5A}$ \\
                        $a_{-2},a_{-1},a_0,a_1,a_2,s$ \\
                        $ $
                } & \specialcell[1.2]{l}{
                        $a_is = \frac{7}{2^{11}}(a_{i+1}+a_{i-1}-2a_i)+\frac{7}{32}s$ \\
                        $ss = \frac{35}{2^{17}}(a_{-2}+a_{-1}+a_0+a_1+a_2)$ \\
                        $(a_i,a_j) = \frac{3}{128}$
                }\\\hline
                \specialcell[1.2]{l}{
                        $\sh{6A}$ \\
                        $a_{-2},a_{-1},a_0,a_1,a_2,a_3,s,\bar s_2$ \\
                        $\bar s_2 = \frac{1}{2}(s_2 + s_2^f)$ \\
                        $\la a_i,a_{i+3}\ra\cong\sh{2A}$ \\
                        $\la a_i,a_{i+2}\ra\cong\sh{3A}$ \\
                        \\ \\
                        $ $
                } & \specialcell[1.2]{l}{
                        $a_ia_{i+2} = \frac{1}{64}(3(a_i+a_{i+2})+a_{i+4}-a_{i+1}-a_{i+3}-a_{i+5}) + \bar s_2$ \\
                        $a_ia_{i+3} = \frac{1}{12}(2(a_i+a_{i+3})-a_{i+1}-a_{i+2}-a_{i+4}-a_{i+5}) - 8s - \frac{8}{3}\bar s_2 $ \\
                        $a_is = \frac{1}{2^{11}}(7(a_{i+1}+a_{i-1})-20a_i)+\frac{7}{32}s$ \\
                        $a_i\bar s_2 = \frac{5}{2^{10}3}a_0
                                - \frac{23}{2^{11}3}(a_{i+1}+a_{i-1})
                                + \frac{7}{2^93}(a_{i+2}+a_{i+4}-\frac{1}{2}a_{i+3})
                                - \frac{3}{32}s
                                + \frac{7}{48}\bar s_2$ \\
                        $ss = \frac{49}{2^{17}3}(a_{-2}+a_{-1}+a_0+a_1+a_2+a_3)
                                - \frac{17}{2^{10}}s
                                - \frac{7}{6144}\bar s_2$ \\
                        $s\bar s_2 = \frac{-21}{2^{17}}(a_{-2}+a_{-1}+a_0+a_1+a_2+a_3)
                                + \frac{21}{2^{10}}s
                                + \frac{15}{2^{11}}\bar s_2$\\
                        $\bar s_2\bar s_2 = \frac{107}{2^{17}}(a_{-2}+a_{-1}+a_0+a_1+a_2+a_3)
                                -\frac{9}{2^{10}}s
                                -\frac{77}{2^{11}}\bar s_2$\\
                        $(a_i,a_{i+1}) = \frac{5}{256}$
                }\\
                \hline
        \end{tabular}
        \caption{The Norton-Sakuma algebras}
                \label{tbl-NS}
        \end{center}
        \end{table}

\section{The cover of $\sh{5A}$}
        \label{sec-5A}

        Only the first factor in \eqref{eq-mother} is $0$ in $\sh{5A}$
        and therefore this factor is in $J_{\sh{5A}}$.
        Thus in $R/J_{\sh{5A}}$, as $\al-1$ is invertible,
        \begin{equation}
                \label{eq-lm-5A}
                \lm=\frac{(2\al-2\bt-1)\bt}{(\al-1)}.
        \end{equation}

        After substituting \eqref{eq-lm-5A},
        the only terms with nonzero coefficients in $a_0(z_2z_2)$
        are $a_{-2},a_0,a_2$,
        and the coefficient of $a_0$ in $a_0(z_2z_2)$ is
        \begin{equation}
                \begin{aligned}
                & \frac{        (\al-2\bt)^2        }{(\al-1)^3(\al-4\bt)^2(\al-\bt)} \\
                & \quad \cdot (2\al-1)(4\al^2-5(2\bt+1)\al+6\bt+1)
                ((\al-1)\lm_2-2\al\bt+\bt(2\bt+1)) \\
                & \quad \cdot (1 + \lm_2 + 4\bt - 4\al - 8\bt^{2} - 2\al\lm_2 + 3\al^{2} - 16\bt^{3} + 24\al\bt^{2} + \al^{2}\lm_2 - 8\al^{2}\bt).
                \end{aligned}
        \end{equation}
        We deduce several possibilities.
        We cannot have a nontrivial relation among the elements $a_{-2},a_0,a_2$
        in the cover of $\sh{5A}$,
        because they are linearly independent in $\sh{5A}$.
        Among the coefficients, the possibilities are as follows.
        Firstly, $\bt = \frac{4\al^2-5\al+1}{2(5\al-3)}$.
        Secondly, $\lm_2 = \frac{(2\al-2\bt-1)\bt}{\al-1} = \lm$.
        Thirdly, $\lm_2 = \frac{8\al^2\bt-24\al\bt^2+16\bt^3-3\al^2+8\bt^2+4(\al-\bt)-1}{(\al-1)^2}$.
        In $\sh{5A}$, only the second possibility is satisfied,
        so in the cover of $\sh{5A}$,
        we have $\lm_2 = \frac{(2\al-2\bt-1)\bt}{\al-1} = \lm$.
        Then the coefficient of $a_2$ in $a_0(z_2z_2)$ becomes
        \begin{equation}
                \begin{aligned}
                \frac{1}{4} & \frac{-\bt(2\bt-1)(3\al-4\bt-1)(-1 - 8\bt + 5\al)(2\bt - \al - 3\al\bt + \al^{2})}{(\al-1)^{3}(\al-4\bt)^{2}} \\
                        & \quad\cdot (-2\bt + \al - 8\bt^{2} + 11\al\bt - 4\al^{2} + 16\al\bt^{2} - 13\al^{2}\bt + 3\al^{3}).
                \end{aligned}
        \end{equation}
        We again have new possibilities.

        If $\bt = \frac{1}{4}(3\al-1)$,
        then $\al\neq1,0,\bt,2\bt,4\bt$ and $\bt\neq0,1$
        means $\al\neq-1,-\frac{1}{3},0,\frac{1}{2},1,\frac{5}{3}$.
        This leads to $\sh{3C_\bt}$,
        as it turns out:
        after substitution, $\lm = \frac{1}{8}(3\al-1) = \frac{1}{2}\bt$,
        and by the fusion rules $A^{a_0}_\al$ is killed.
        But we do not pursue this,
        since it is not satisfied by $(\al,\bt)=(\frac{1}{4},\frac{1}{32})$.

        The possibilities $\bt = \frac{\al(1-\al)}{2-3\al}$
        and $(16\al-8)\bt^2+(-13\al^2+11\al-2)\bt+3\al^3-4\al^2+\al=0$
        are not in $R/J_{\sh{5A}}$.

        If $\bt = \frac{1}{8}(5\al-1)$,
        then $\al\neq1,0,\bt,2\bt,4\bt$ and $\bt\neq0,1$
        means $\al\neq-\frac{1}{3},0,\frac{1}{5},\frac{1}{3},1,\frac{9}{5}$.
        This is satisfied by $(\al,\bt)=(\frac{1}{4},\frac{1}{32})$.
        We now compute the action of $a_2$
        on the subspace $Q$ spanned by
        \begin{equation}
                g = s_2 - s_2^f,
                \quad
                h = \bt(a_{-2}+a_{-1}+a_0+a_1+a_2) + \frac{2}{\bt}s + \frac{1}{\bt}(s_2+s_2^f).
        \end{equation}
        $Q$ is fixed by $\ad(a_2)$, and
        \begin{equation}
                        \renewcommand{\arraystretch}{0.9}
                N = \ad(a_2)\vert_Q =
                        \frac{1}{8(\al-1)}
                        \begin{pmatrix}
                                9\al^2-1 & 9\al^2-1 \\
                                \frac{3 - 27\al + 57\al^{2} - 17\al^{3}}{1 - 3\al} & -\al^2-8\al+1
                        \end{pmatrix}.
        \end{equation}
        Therefore $Q$ must decompose into a direct sum
        of $0,\al$ or $\bt$-eigenvectors for $a_2$.
        Now suppose that $\mu$ is an eigenvalue of $N$,
        so that $\det(N-\mu I_2) = 0$.
        That is,
        \begin{equation}
                \frac{1}{32}\frac{(1 - 4\al + 32\mu^{2} - 32\al\mu - 11\al^{2} - 32\al\mu^{2} + 32\al^{2}\mu + 30\al^{3})}{(1 - \al)}
        \end{equation}
        We can substitute $\mu=0,\al,\bt = \frac{1}{8}(5\al-1)$ respectively
        and solve the resulting equation.
        If $\mu = 0$,
        then the numerator factors as $(\al-1)(2\al-1)(3\al+1)(5\al-1)$;
        the only legitimate possibility is $\al = \frac{1}{2}$ with $\bt = \frac{3}{16}$,
        which will be a special case for us.

        If $\mu = \al$,
        the numerator becomes $-60\al^4+74\al^3-5\al^2-10\al+2$.

        If $\mu = \frac{1}{8}(5\al-1)=\bt$,
        the numerator becomes $-\frac{1}{64}(5\al-1)(768\al^3-832\al^2-69\al+129)$.

        Note that $\al = \frac{1}{5}$ was already ruled out.
        So we are left with the additional assumptions that
        \begin{equation}
                \label{eq-5A-forbidden}
                (-60\al^4+74\al^3-5\al^2-10\al+2)
                (768\al^3-832\al^2-69\al+129) \neq 0.
        \end{equation}
        When additionally $(\al,\bt)\neq(\frac{1}{2},\frac{3}{16})$,
        we have to quotient by $Q$.
        And then get the multiplication in Table \ref{tbl-5A}:

        \begin{theorem}
                \label{thm-5A}
                The algebra $A$ over $R = \ZZ[\sfrac{1}{2},\al]$
                with basis $a_{-2},a_{-1},a_0,a_1,a_2,s$
                and multiplication from Table \ref{tbl-5A}
                is a $2$-generated $\Phi(\al,\frac{1}{8}(5\al-1))$-axial algebra
                for all $\al\not\in\{1,0,\frac{1}{5},\frac{1}{2},\frac{9}{5}\}$
                nor a root of \eqref{eq-5A-forbidden},
                and is the cover $\sh{5A_\al}$ of $\sh{5A}$.
        \end{theorem}
        \begin{proof}
                $A$ is generated by $a_0,a_1$,
                so we only need to check the fusion rules for $a_0,a_1$.
                We do this in \cite{gap},
                using Lemma \ref{lem-eigvect-ising}
                and the obvious symmetry between the $a_i$.

                We check manually, that is, in \cite{gap},
                that in the case $(\al,\bt) = (\frac{1}{2},\frac{3}{16})$,
                the vector space $Q$ is also killed
                and the resulting algebra has the same presentation.
        \end{proof}

        \begin{table}[h]
        \begin{center}
        \renewcommand{\arraystretch}{1.2}
        \begin{tabular}{ll}
                \hline Description \quad & Products \\
                \hline
                $a_{-2},a_{-1},a_0,a_1,a_2,s$ &
                        $a_ia_{i+2} = \frac{1}{16}(5\al-1) (a_i + a_{i+2} - a_{i+1} - a_{i+3} - a_{i+4}) - s$ \\
                $a_i = a_{i\bmod5}$ &
                        $a_is = \frac{1}{64}(1 - 5\al)(1 + 3\al) a_i - \frac{1}{2^{7}}(1 + 3\al)(1 - 5\al) (a_{i-1} + a_{i+1}) + \frac{1}{8}(1 + 3\al) s$ \\
                 &
                  $ss = \frac{1}{2^{11}}(7\al-3)(1 + 3\al)(1 - 5\al) (a_{-2} + a_{-1} + a_0 + a_1 + a_2) + \frac{5}{2^{7}}(1 - 5\al)(1 + 3\al) s$ \\
                 &
                        $(a_i,a_j) = \frac{3}{32}(5\al-1)$ \\
                \hline
        \end{tabular}
        \caption{$\sh{5A_\al}$}
                \label{tbl-5A}
        \end{center}
        \end{table}

        \begin{lemma}
                \label{lem-5A-form}
                The algebra $\sh{5A_\al}$ is Frobenius;
                it is positive-definite for $\frac{1}{5}<\al<\frac{7}{3}$,
                and degenerate for $\al = -\frac{1}{3},\frac{7}{3}$.
        \end{lemma}
        \begin{proof}
                We check by direct computation in \cite{gap}
                that the form on $\sh{5A_\al}$
                induced as a quotient of the form in Section \ref{sec-mt}
                is Frobenius.
                We also use \cite{gap}
                to calculate the Gram matrix,
                the determinant of which is
                \begin{equation}
                        \frac{-5^6}{2^{36}}(3\al-7)^5(3\al+1)^2(5\al-1).
                \end{equation}
                This is positive exactly when $(3\al-7)(5\al-1)<0$,
                and the root $(5\al-1)$ is not admissible
                since this implies $\bt = \frac{1}{5} = \al$,
                but the other roots have $(\al,\bt)$
                as $(\sfrac{7}{3},\sfrac{4}{3})$ and $(\sfrac{1}{3},\sfrac{1}{12})$.
        \end{proof}

        \begin{lemma}
                \label{lem-5a-miyam}
                We have $\size{\tau(a_0)\tau(a_1)} = 5$.
        \end{lemma}
        \begin{proof}
                $\tau(a_0)$ and $\tau(a_1)$ act as the permutation matrices
                corresponding to $(1,5)(2,4)(3)(6)$ and $(1,2)(3,5)(4)(6)$
                on $\sh{5A_\al}$, respectively.
                Since $a_0$ and $a_1$ are in the same orbit under $T$,
                we have that $\size{\tau(a_0)\tau(a_1)} = 5$ everywhere by Lemma \ref{lem-global-order}.
        \end{proof}

\section{The cover of $\sh{6A}$}
        \label{sec-6A}

        Only the second factor in \eqref{eq-mother} is $0$ in $\sh{6A}$:
        \begin{equation}
                \label{eq-6A-mother}
                2(\al(\al-1)-6\al\bt+4\bt)\lm+(10\bt+1)\al\bt-2\bt^2(2+3\bt) = 0.
        \end{equation}
        Since $\dim\sh{6A}=8$,
        the key relations for the cover come from the ring,
        that is, $\bar I_{\sh{6A}} = \bar J_{\sh{6A}}\hat U$.
        Suppose that the coefficient of $\lm$ in \eqref{eq-6A-mother} is $0$,
        that is, $\al^2-6\al\bt-\al+4\bt = 2(2-3\al)\bt + \al(1-\al) =0$.
        As $\al=\frac{2}{3}$ is not a solution to this equation,
        we can rearrange to find
        \begin{equation}
                \bt = \frac{\al(\al-1)}{2(2-3\al)}
        \end{equation}
        (which is not in $\bar J_{\sh{6A}}$)
        and substituting this into \eqref{eq-6A-mother},
        \begin{equation}
                7\al^2-12\al+4 = 0.
        \end{equation}
        This irreducible polynomial is coprime to $\al-\frac{1}{4}$,
        so that $(\al-\sfrac{1}{4})(7\al^2-12\al+4)=(1)$
        and therefore if $7\al^2-12\al+4\in J_{\sh{6A}}$
        then $J_{\sh{6A}}+(\al-\sfrac{1}{4})+(\bt-\sfrac{1}{32}) = (1)$ is not maximal.
        So $\al$ is not a root of $7\al^2-12\al+4$,
        that is, $\al\neq\frac{1}{7}(4\pm\sqrt2)$.

        Therefore $\al^2-6\al\bt-\al+4\bt\neq0$ and
        \begin{equation}
                \lm=\frac{(10\bt^2+\bt)\al-(4\bt+6)\bt^2}{2(\al^2-6\al\bt-\al+4\bt)}.
        \end{equation}
        After making the substitution,
        we again calculate $a(z_2z_2)$,
        and find that the coefficient of $a_2$ is
        \begin{equation}
                \begin{aligned}
                        \label{eq-6A-lm2}
                        & \frac{\bt(\al-\bt)(1-2\bt)}
                        {4(\al-4\bt)(\al^2-(6\bt+1)\al+4\bt)^2)}\cdot\\
                        & \Bigl(
        -2\al(\al^{2} - (6\bt + 1)\al + 4\bt)^{2}\lm_2 + 3\al^{6} - (3\bt + 4)\al^{5} - (60\bt^{2} + 6\bt - 1)\al^{4} + \bt(2\bt + 1)(98\bt + 9)\al^{3} \\
        & \quad\quad - 2\bt(88\bt^{3} + 2^{2}3\cdot11\bt^{2} + 30\bt + 1)\al^{2} + 8\bt^{2}(8\bt^{3} + 20\bt^{2} + 12\bt + 1)\al - 32\bt^{5} - 32\bt^{4} - 8\bt^{3}
                        \Bigr).
                \end{aligned}
        \end{equation}
        We deduce an expression for $\lm_2$
        (its coefficient being not $0$
        when $\al\neq\frac{1}{7}(4\pm\sqrt2)$).
        Substituting, we find that the coefficient of $a_2$ in $a_0(xx_2)$ is
        \begin{equation}
                \frac{-\bt^2(2\bt-1)\al(\al-2\bt)(\al^2+8\bt\al-4\bt)}{4(\al-4\bt)(\al^2-(6\bt+1)\al+4\bt)}.
        \end{equation}
        So $\bt = \frac{\al^2}{4(1-2\al)}$.
        After specialising,
        we find no more relations.
        We get

        \begin{theorem}
                \label{thm-6A}
                The algebra $A$ over $R = \ZZ[\sfrac{1}{2},\al,\al^{-1},(1-2\al)^{-1},(1-3\al)^{-1},(2-5\al)^{-1}]$,
                $\al\not\in\{0,-4\pm2\sqrt5,\frac{1}{3},\frac{2}{5},\frac{1}{2},1\}$,
                with basis $a_{-2},a_{-1},a_0,a_1,a_2,s,s_2,s_2^f$
                and multiplication from Table \ref{tbl-6A}
                is a $2$-generated $\Phi(\al,\frac{\al^2}{4(1-2\al)})$-axial algebra
                and the cover $\sh{6A_\al}$ of Norton-Sakuma algebra $\sh{6A}$.
        \end{theorem}
        \begin{proof}
                $A$ is generated by $a_0,a_1$,
                so we only need to check the fusion rules for $a_0,a_1$.
                We do this in \cite{gap},
                using Lemma \ref{lem-eigvect-ising}.
        \end{proof}

        \begin{table}[h]
        \begin{center}
        \renewcommand{\arraystretch}{1.2}
        \begin{tabular}{ll}
                \hline Description \quad & Products \\
                \hline
                $a_{-2},a_{-1},a_0,a_1,a_2,a_3,s,\bar s_2$ &
                        $a_ia_{i+2} = \frac{1}{8}\frac{\al(3\al - 1)}{(2\al - 1)} a_{i+4} - \frac{1}{8}\frac{\al(3\al - 1)}{(2\al - 1)} (a_{i+1} + a_{i+3} + a_{i+5}) + \frac{1}{8}\frac{(\al - 1)\al}{(2\al - 1)} (a_i + a_{i+2}) + \bar s_2$ \\
                $\bar s_2 = \frac{1}{2}(s_2 + s_2^f)$ &
                        $a_ia_{i+3} = \frac{1}{2}\frac{\al(3\al - 1)}{(5\al - 2)}(a_i+a_{i+3}) - \frac{1}{2}\frac{\al(2\al - 1)}{(5\al - 2)} (a_{i+1} + a_{i+2} + a_{i+4} + a_{i+5}) + 4\frac{(2\al - 1)}{\al} s$ \\
                $a_i = a_{i\bmod6}$ &
                        $\qquad\qquad - 4\frac{2\al-1}{5\al-2}\bar s_2$ \\
                $\la a_i,a_{i+3}\ra\cong\sh{3C_\al}$ &
                  $a_is = -\frac{1}{32}\frac{\al^{3}(9\al - 4)}{(2\al - 1)^{2}} (a_{i-1} + a_{i+1}) - \frac{1}{16}\frac{\al^{2}(3\al - 2)}{(2\al - 1)} a_i + \frac{1}{4}\frac{\al(9\al - 4)}{(2\al - 1)} s$ \\
                $\la a_i,a_{i+2}\ra\cong\sh{3A_{\al,\bt}}$ &
                  $a_i\bar s_2 =\frac{1}{32}\frac{\al^{2}(9\al - 4)}{(5\al - 2)} (a_{i+4} + a_{i+2}) - \frac{1}{32}\frac{\al^{2}(3\al - 1)(43\al^{2} - 37\al + 8)}{(2\al - 1)^{2}(5\al - 2)} (a_{i-1} + a_{i+1})$ \\
                 &
                         $\qquad\qquad - \frac{1}{32}\frac{\al(87\al^{3} - 3^{3}5\al^{2} + 60\al - 8)}{(2\al - 1)(5\al - 2)} a_i - \frac{1}{32}\frac{\al^{2}(3\al - 1)(9\al - 4)}{(2\al - 1)(5\al - 2)} a_{i+3} + \frac{1}{4}\frac{(3\al - 1)(5\al - 2)}{(2\al - 1)} s$ \\
                 &
                         $\qquad\qquad + \frac{1}{4}\frac{\al(9\al - 4)}{(5\al - 2)} \bar s_2$ \\
                 &
                  $ss = -\frac{1}{2^{8}}\frac{\al^{4}(3\al - 1)(9\al - 4)^{2}}{(2\al - 1)^{3}(5\al - 2)} (a_i + a_{i+1} + a_{i+2} + a_{i+3} + a_{i+4} + a_{i+5})$ \\
                 & 
                         $\qquad\qquad+ \frac{1}{16}\frac{\al^{2}(39\al^{2} - 22\al + 2)}{(2\al - 1)^{2}} s - \frac{1}{32}\frac{\al^{4}(9\al - 4)}{(2\al - 1)^{2}(5\al - 2)} \bar s_2$ \\
                 &
                  $s\bar s_2 = -\frac{1}{2^{8}}\frac{\al^{3}(3\al - 1)(5\al - 2)(9\al - 4)}{(2\al - 1)^{3}} (a_i + a_{i+1} + a_{i+2} + a_{i+3} + a_{i+4} + a_{i+5})$ \\
                 & 
                         $\qquad\qquad+ \frac{3}{16}\frac{\al^{2}(3\al - 1)(9\al - 4)}{(2\al - 1)^{2}} s + \frac{1}{32}\frac{\al^{2}(3\al - 2)(5\al - 2)}{(2\al - 1)^{2}} \bar s_2$ \\
                 &
                  $\bar s_2\bar s_2 = -\frac{1}{2^{8}}\frac{\al^{2}(3\al - 1)}{(2\al - 1)^{3}(5\al - 2)}(3\cdot7\cdot29\al^{4} - 2^{3}131\al^{3} + 2^{2}173\al^{2} - 2^{4}13\al + 24)\cdot$ \\
                 &
                         $\qquad\qquad(a_i + a_{i+1} + a_{i+2} + a_{i+3} + a_{i+4} + a_{i+5}) + \frac{3}{16}\frac{\al(3\al - 1)^{2}(5\al - 2)}{(2\al - 1)^{2}} s$ \\
                 &
                         $\qquad\qquad+ \frac{1}{32}\frac{\al(3\cdot173\al^{4} - 2^{6}13\al^{3} + 2\cdot3\cdot5\cdot17\al^{2} - 2^{4}3^{2}\al + 16)}{(2\al - 1)^{2}(5\al - 2)} \bar s_2$ \\
                 &
                        $(a_i,a_{i+1}) = \frac{1}{16}\frac{\al^{2}(2-3\al)}{(2\al - 1)^{2}}$ \\
                \hline
        \end{tabular}
        \caption{$\sh{6A_\al}$}
                \label{tbl-6A}
        \end{center}
        \end{table}

        \begin{lemma}
                \label{lem-6A-form}
                The algebra $\sh{6A_\al}$ is Frobenius;
                it is positive-definite for $\al\in(\frac{1}{2},\frac{4}{7})\cup(\frac{2}{3},1)\cup(\frac{1}{24}(1-\sqrt{97}),\frac{1}{24}(1+\sqrt{97}))$,
                and degenerate for $\al = \frac{3}{2},\frac{4}{7},\frac{1}{24}(1\pm\sqrt{97})$.
        \end{lemma}
        \begin{proof}
                We check by direct computation in \cite{gap}
                that the form on $\sh{5A_\al}$
                induced as a quotient of the form in Section \ref{sec-mt}
                is Frobenius.
                We also use \cite{gap}
                to calculate the Gram matrix,
                the determinant of which is
                \begin{equation}
                        \frac{-\al^8}{2^{31}(2\al-1)^{17}}(\al-1)^3(3\al-2)(3\al-1)^2(5\al-2)^2(7\al-4)^5(\al^2+4\al-2)^4(12\al^2-\al-2).
                \end{equation}
                This is positive exactly when $(2\al-1)(\al-1)(3\al-2)(7\al-4)(12\al^2-\al-2)<0$.
                The roots of the equation which are not already ruled out in Theorem \ref{thm-6A}
                correspond to $(\al,\bt)$
                being $(\sfrac{3}{2},\sfrac{-9}{32}),
                (\sfrac{4}{7},\sfrac{-4}{7})$
                and $\al=\sfrac{1}{24}(1\pm\sqrt{97})$.
        \end{proof}

        \begin{lemma}
                \label{lem-6a-miyam}
                On $\sh{6A_\al}$,
                $\size{\tau(a_0)\tau(a_1)} = 3$,
                the flip is an automorphism
                and $\size{\tau(a_0)\tau(a_1)}\leq6$ in any larger algebra.
        \end{lemma}
        \begin{proof}
                On $\sh{6A_\al}$,
                $\tau(a_0)$ is the permutation matrix of $(1,5)(2,4)(3)(6)(7)(8)$
                and $\tau(a_1)$ fixes $a_1$ and $a_{-2},s,s_2,s_2^f$,
                $a_{-1}$ is mapped to
                \begin{equation}
                        a_3 = a_{-2}+a_0+a_2 - a_{-1} - a_1 + 4\frac{(1-2\al)}{\al(3\al-1)}(s_2-s_2^f),
                \end{equation}
                and swaps $a_0$ with $a_2$.
                Therefore, for $\kp = 4\frac{1-2\al}{\al(3\al-1)}$,
                \begin{equation}
                        \renewcommand{\arraystretch}{0.9}
                        \tau(a_0)\tau(a_1) =
                        \begin{pmatrix}
                                 0 & 0 & 1 & 0 & 0 & 0 & 0 & 0 \\
                                 0 & 0 & 0 & 1 & 0 & 0 & 0 & 0 \\
                                 0 & 0 & 0 & 0 & 1 & 0 & 0 & 0 \\
                                 1 & -1 & 1 & -1 & 1 & 0 & \kp & -\kp \\
                                 1 & 0 & 0 & 0 & 0 & 0 & 0 & 0 \\
                                 0 & 0 & 0 & 0 & 0 & 1 & 0 & 0 \\
                                 0 & 0 & 0 & 0 & 0 & 0 & 1 & 0 \\
                                 0 & 0 & 0 & 0 & 0 & 0 & 0 & 1
                        \end{pmatrix},
                \end{equation}
                and $(\tau(a_0)\tau(a_1))^3$ is the identity matrix.
                By Lemma \ref{lem-global-order},
                $\tau(a_0)\tau(a_1)$ has order at most $\size{a_0^T\cup a_1^T} = 6$.
        \end{proof}

\section{The covers of $\sh{4A},\sh{4B}$}
        \label{sec-4X}

        The third factor in \eqref{eq-mother} is $0$ in $\sh{4A}$ and in $\sh{4B}$
        and therefore, in these cases, as $\al\neq1$,
        \begin{equation}
                \label{eq-case-3}
                \lm_2 = \frac{1}{\al\bt}(4(2\al-1)\lm^2-2(\al^2+6\bt\al-4\bt)\lm+2\al^2\bt+4(\al-1)\bt^2).
        \end{equation}

        \begin{lemma}
                \label{lem-4B-badman}
                If $(\al,\bt,\lm)$ is not a root of \eqref{eq-4B-badman},
                then $a_0 = a_4$.
        \end{lemma}
        \begin{proof}
                Our aim is first to show that
                $a_2 = a_{-2}$,
                by showing that
                \begin{equation}
                        k = a_2 - a_{-2} \in A^{a_0}_{1,0,\al}\oplus\la kk\ra,
                \end{equation}
                that is,
                we find constants $\kp,\kp_1,\kp_z,\kp_{zz},\kp_{z_2},\kp_x,\kp_{x_2}$
                such that
                \begin{equation}
                        \kp k = \kp_1 a_0 + \kp_z z + \kp_{zz} zz + \kp_{z_2} z_2 + \kp_x x + \kp_{x_2} x_2 + kk,
                \end{equation}
                which implies that $\kp k = 0$,
                because $k\in A^{a_0}_\bt$
                so $kk\in A^{a_0}_{1,0,\al}$
                and the semisimplicity of $a_0$ implies that $A^{a_0}_\bt\cap A^{a_0}_{1,0,\al} = 0$.

                We multiply $kk$ using the multiplication table.
                That suitable $\kp,\kp_1,\kp_z,\kp_{zz},\kp_{z_2},\kp_x,\kp_{x_2}$ exist
                follows from the fact that
                the coefficients of $a_{-1}$ and $a_1$ in $a_2a_{-2}$ are equal,
                so that there are only $7$ parameters;
                now the chosen set is linearly independent in those seven variables.
                Even though the actual equations are lengthy,
                the working is straightforward:

                \begin{align*}
                        & \kp_{zz} = - \frac{kk\vert_{s_2^f}}{zz\vert_{s_2^f}}, \\
                        & kk' = kk + \kp_{zz}zz, \\
                        & \kp_{z_2} - \kp_{x_2} = kk'\vert_{s_2}, \\
                        & \frac{1}{2}(\al-\bt)\kp_{z_2} + \frac{1}{2}\bt\kp_{x_2} = -\frac{1}{2}(kk'\vert_{a_{-2}} + kk'\vert_{a_2}), \\
                        & \kp_{x_2} = \frac{2}{\al}(-\frac{1}{2}(kk'\vert_{a_{-2}} + kk'\vert_{a_2})-\frac{1}{2(\al-\bt)}kk'\vert_{s_2}), \\
                        & \kp_{z_2} = \kp_{x_2} + kk'\vert_{s_2}, \\
                        & kk'' = kk + \kp_{z_2}z_2 + \kp_{x_2}x_2, \\
                        & \kp_z - \kp_x = kk''\vert_s, \\
                        & \frac{1}{2}(\al-\bt)\kp_z + \frac{1}{2}\bt\kp_x = -\frac{1}{2}(kk''\vert_{a_{-1}} + kk''\vert_{a_1}), \\
                        & \kp_x = \frac{2}{\al}(-\frac{1}{2}(kk''\vert_{a_{-1}} + kk''\vert_{a_1})-\frac{1}{2(\al-\bt)}kk''\vert_s), \\
                        & \kp_z = \kp_x + kk''\vert_s, \\
                        & kk''' = kk'' + \kp_z z + \kp_x x, \\
                        & \kp_1 = kk'''\vert_{a_0}.
                \end{align*}
                Note that the only division is by $zz\vert_{s_2^f} = -2\al(\al-\bt)(\al-2\bt)$,
                which is invertible by assumption.
                Then,
                \begin{equation}
                        \label{eq-4B-badman}
                        \begin{aligned} \kp = \;
                                & \frac{-2(-2(\al^{2} - 6\al\bt - \al + 4\bt)\lm + 10\al\bt^{2} - 4\bt^{3} + \al\bt - 6\bt^{2})}{\al\bt^2(\al-\bt)^2(\al-2\bt)(\al-4\bt)^3}\cdot \\
& \Bigl(-(8(\al - 1)(\al - 2\bt)^{2}(3\al^{2} - 8\al\bt - \al + 4\bt)\lm^{2} - 2(4\al^{6} - \al^{5}\bt - 96\al^{4}\bt^{2} + 2^{2}73\al^{3}\bt^{3} \\
& - 2^{8}\al^{2}\bt^{4} - 6\al^{5} + 7\al^{4}\bt + 2\cdot3\cdot19\al^{3}\bt^{2} - 2^{7}3\al^{2}\bt^{3} + 2^{4}23\al\bt^{4} + 2\al^{4} - 2\al^{3}\bt - 40\al^{2}\bt^{2} \\
& + 2^{7}\al\bt^{3} - 2^{7}\bt^{4})\lm + 8\al^{6}\bt - 22\al^{5}\bt^{2} - 20\al^{4}\bt^{3} + 2^{2}3\cdot11\al^{3}\bt^{4} - 2^{3}3\cdot5\al^{2}\bt^{5} - 32\al\bt^{6} \\
& - 8\al^{5}\bt + 5\al^{4}\bt^{2} + 2^{2}31\al^{3}\bt^{3} - 2^{4}23\al^{2}\bt^{4} + 2^{5}11\al\bt^{5} + 2\al^{4}\bt + 4\al^{3}\bt^{2} - 56\al^{2}\bt^{3} + 2^{7}\al\bt^{4} \\
& - 2^{7}\bt^{5})\lm_2 + 16(\al - 1)(\al^{4} - 8\al^{3}\bt + 12\al^{2}\bt^{2} - \al^{3} + 14\al^{2}\bt - 16\al\bt^{2} - 4\al\bt + 4\bt^{2})\lm^{3} \\
& - 4(3\al^{5}\bt - 60\al^{4}\bt^{2} + 2^{2}3\cdot11\al^{3}\bt^{3} -48\al^{2}\bt^{4} + 3\al^{5} - 17\al^{4}\bt + 2\cdot103\al^{3}\bt^{2} - 2^{3}3\cdot13\al^{2}\bt^{3} \\
& + 64\al\bt^{4} - 5\al^{4} + 28\al^{3}\bt - 2^{2}53\al^{2}\bt^{2} + 2^{2}5\cdot11\al\bt^{3} - 16\bt^{4} + 2\al^{3} - 12\al^{2}\bt + 64\al\bt^{2} - 48\bt^{3})\lm^{2} \\
& - 2\bt(14\al^{5}\bt - 8\al^{4}\bt^{2} - 56\al^{3}\bt^{3} + 32\al^{2}\bt^{4} - 11\al^{5} + 33\al^{4}\bt - 2\cdot3^{2}17\al^{3}\bt^{2} + 2^{3}59\al^{2}\bt^{3} -\\
& 2^{7}\al\bt^{4} + 14\al^{4} - 80\al^{3}\bt + 2^{6}7\al^{2}\bt^{2} - 2^{6}7\al\bt^{3} + 64\bt^{4} - 6\al^{3} + 40\al^{2}\bt - 2^{3}19\al\bt^{2} + 96\bt^{3})\lm\\
& + 2\al^{5}\bt^{2} - 28\al^{4}\bt^{3} - 60\al^{3}\bt^{4} + 2^{3}17\al^{2}\bt^{5} - 32\al\bt^{6} + 5\al^{4}\bt^{2} - 48\al^{3}\bt^{3} + 2^{2}7\cdot11\al^{2}\bt^{4}\\
& - 2^{6}5\al\bt^{5} + 64\bt^{6} - 4\al^{3}\bt^{2} + 32\al^{2}\bt^{3} - 2^{4}7\al\bt^{4} + 64\bt^{5}\Bigr).
        \end{aligned}
                \end{equation}
                and it's an irreducible polynomial.
                So assuming that $(\al,\bt,\lm)$ are not a root of \eqref{eq-4B-badman},
                we have that $a_2 = a_{-2}$.

                By applying $\tau(a_1)$ to both sides,
                it follows that $a_2^{\tau(a_1)} = a_{-2}^{\tau(a_1)}$,
                that is, $a_0 = a_4$.
        \end{proof}

        Now $a_2^{\tau(a_0)} = a_{-2} = a_2$,
        so that $a_2\in A^{a_0}_{1,0,\al}$,
        and likewise $a_0\in A^{a_2}_{1,0,\al}$
        and $a_{-1}\in A^{a_1}_{1,0,\al},a_1\in A^{a_{-1}}_{1,0,\al}$.
        Therefore we may apply Theorem \ref{thm-hrs-sak}
        to the subalgebras $\la a_0,a_2\ra, \la a_1,a_{-1}\ra$.
        In particular,
        either $\lm_2 = \frac{\al}{2}$
        or $\lm_2 = 0$,
        corresponding to $\lm_2 - \frac{\al}{2}\in J_{\sh{4B}}$
        and $\lm_2 \in J_{\sh{4A}}$ respectively.
        We take the two cases separately.

        \begin{lemma}
                \label{lem-4B-lm}
                $\bt - \frac{\al^2}{2}$ and $\lm - \frac{\bt}{2}$ are in $J_{\sh{4B}}$.
        \end{lemma}
        \begin{proof}
                After making the substitution $a_{-2} = a_2$,
                we calculate $0 = a_0(zz_2) = $
                \begin{equation}
                        \begin{aligned}
                                = \frac{1}{4}\frac{(\al - \bt)}{(\al - 4\bt)}
                                \Bigl( (8(3\bt - 1)(2\al - 1)\lm^{2} - 2(3(4\bt - 1)\al^{2} + (2\bt - 1)(10\bt - 1)\al - 12\bt^{2} + 4\bt)\lm & \\
                                + 2\bt\al^{3} + \bt(11\bt - 3)\al^{2} + \bt(8\bt^{2} - 12\bt + 1)\al - 4\bt^{3} + 2\bt^{2})\Bigr) (a_{-1} - a_1)&,
                        \end{aligned}
                \end{equation}
                and in particular there is no linear dependence between $a_{-1}$ and $a_1$ in $\sh{4B}$,
                so the coefficient must be $0$.
                Therefore we have a quadratic formula for $\lm$.
                We also have a different quadratic formula for $\lm$
                from combining \eqref{eq-case-3} and $\lm_2 = \frac{\al}{2}$.
                The resultant of the numerators of the two quadratics
                with respect to $\lm$ is
                \begin{equation}
                        8(2\al-1)\bt(2\bt-1)(2\bt-\al^2)(2\bt+\al^2-1)(4\al^{3} - (18\bt + 1)\al^{2} + 4\bt(4\bt + 1)\al - 4\bt^{2}),
                \end{equation}
                and the only factor which is $0$ when $(\al,\bt)=(\frac{1}{4},\frac{1}{32})$
                is $2\bt-\al^2$,
                that is, $\bt = \frac{\al^2}{2}$.
                Now substituting $\bt = \frac{\al^2}{2}$
                back into both of the quadratic formulae for $\lm$
                yields that the only common factor is $4\lm-2\al^2$,
                which is also the only factor of either which is $0$ in $\sh{4B}$.
                This gives $\lm = \frac{\al^2}{4}$.
        \end{proof}

        \begin{lemma}
                \label{lem-4B-kernel}
                If $\bt = \frac{\al^2}{2}$ and $\lm = \frac{\al^2}{4}$
                then the subspace $K$ spanned by
                \begin{equation}
                        \begin{aligned}
                                y_2 & = a_2 - a_{-2} \in A^{a_0}_\bt \\
                                x & = \al^3(a_0+a_2) + \al^2(a_{-1}+a_1) + 4s + 2\al s_2 \in A^{a_0}_\al \\
                                x^f & = \al^2(a_0+a_2) + \al^3(a_{-1}+a_1) + 4s + 2\al s_2^f \\
                        \end{aligned}
                \end{equation}
                is killed by the fusion rules.
        \end{lemma}
        \begin{proof}
                $K$ is closed under the action of $\ad(a_2)$:
                \begin{equation}
                        \begin{aligned}
                                & (\al-1)(2\al-1)^2a_2y_2 =
                                        \frac{1}{2}(\al-3)(2\al-1)\al^2y_2
                                        - 2 x
                                        + 2(\al-2)(\al-1)x^f \\
                                & \begin{aligned}
                                        (\al-1)(2\al-1)^2a_2x & =
                                        \frac{1}{4}\al^5(\al+1)(2\al-1)y_2
                                        + \al(\al^4-\al^3+3\al^2-3\al+1)x \\
                                        & \qquad - \al(\al-1)^2(\al^2-\al+1)x^f
                                        \end{aligned} \\
                                & \begin{aligned}
                                        (\al-1)(2\al-1)^2a_2x^f & =
                                        -\frac{1}{4}\al^3(2\al-1)(\al^3-6\al^2+6\al-2)y_2 \\
                                        & \qquad -\frac{1}{2}\al(2\al^4-10\al^3+9\al^2-\al-1)x
                                        +\al(\al-1)^3(\al+1)x^f.
                                        \end{aligned} \\
                        \end{aligned}
                \end{equation}
                The determinant of $\ad(a_2)\vert_K-\mu I_3$
                is must be $0$ for $\mu = 0,\al,\frac{\al^2}{2}$,
                since $a_2$ is a $\Phi(\al,\sfrac{\al^2}{2})$-axis
                and $a_2\not\in K$.
                These cases respectively correspond to
                \begin{equation}
                        \label{eq-4B-no-kernel}
                        \begin{gathered}
                                \al(2\al^5-4\al^4-6\al^3+13\al^2-7\al+1) = 0 \\
                                66\al^6-124\al^5+26\al^4+71\al^3-57\al^2+17\al-2 = 0 \\
                                \al(16\al^7-24\al^6-22\al^5+78\al^4-91\al^3+57\al^2-18\al+2) = 0
                        \end{gathered}
                \end{equation}
                with common additional factor $-\frac{\al^3}{4}(\al-1)^{-1}(2\al-1)^{-4}$.
                The equations have three, two and two real roots respectively.
                Note that $\al = \sfrac{1}{4}$ is not a solution to any of them.
                Therefore if $\al$ is not a solution to \eqref{eq-4B-no-kernel}
                then $K$ is killed by the fusion rules.
        \end{proof}

        \begin{theorem}
                \label{thm-4B}
                The algebra $A$, denoted $\sh{4B_\al}$,
                over $R = \ZZ[\sfrac{1}{2},\al,\al^{-1}]$,
                for $\al\neq0,1,2$ nor a root of \eqref{eq-4B-no-kernel},
                with basis $a_{-1},a_0,a_1,a_2,s$
                and multiplication from Table \ref{tbl-4B}
                is a $2$-generated $\Phi(\al,\sfrac{\al^2}{2})$-axial algebra
                and the cover of $\sh{4B}$.
        \end{theorem}
        \begin{proof}
                $A$ is generated by $a_0,a_1$,
                so we only need to check the fusion rules for $a_0,a_1$.
                We do this in \cite{gap}.
                This is easily done using Lemma \ref{lem-eigvect-ising}.
        \end{proof}

        \begin{lemma}
                \label{lem-4B-form}
                The algebra $\sh{4B_\al}$ is Frobenius;
                it is positive-definite for all $\al$,
                and degenerate for $\al = -1$.
        \end{lemma}
        \begin{proof}
                We check by direct computation in \cite{gap}
                that the form on $\sh{4B_\al}$
                induced as a quotient of the form in Section \ref{sec-mt}
                is Frobenius.
                We also use \cite{gap}
                to calculate the Gram matrix,
                the determinant of which is
                \begin{equation}
                        \frac{\al^4}{2^8}(\al-2)^4(\al+1)^2.
                \end{equation}
                This is always positive.
                Its roots are $\al = -1$.
        \end{proof}

        \begin{table}[h]
        \begin{center}
        \renewcommand{\arraystretch}{1.2}
        \begin{tabular}{ll}
                \hline Description \quad & Products \\
                \hline
                $a_{-1},a_0,a_1,a_2,s$ &
                        $a_is = -\frac{1}{4}\al^{2}(1 - \al + \al^{2}) a_{-1} + \frac{1}{8}(2 - \al)\al^{3} (a_0 + a_2) + \frac{1}{2}(2 - \al)\al s$ \\
                $a_i = a_{i\bmod4}$ &
                        $ss = \frac{1}{8}(-2 + \al)(\frac{1}{2}-1 + 2\al)\al^{4} (a_{-1} + a_0 + a_1 + a_2) + \frac{1}{2}\al^{2}(\frac{1}{2}1 - 6\al + 2\al^{2}) s$ \\
                $\la a_i,a_{i+2}\ra\cong\sh{3C_\al}$ &
                        $(a_i,a_j) = \frac{\al^2}{4}$ \\
                \hline
        \end{tabular}
        \caption{$\sh{4B_\al}$}
                \label{tbl-4B}
        \end{center}
        \end{table}


        \begin{lemma}
                \label{lem-4A-dichot}
                $s_2 + \bt(a_0+a_2), s_2^f + \bt(a_1+a_{-1})\in I_{\sh{4A}}$
                and $\lm_1-\bt\in J_{\sh{4A}}$.
        \end{lemma}
        \begin{proof}
                From \eqref{eq-case-3},
                when $\lm_2 = 0$
                we get
                \begin{equation}
                        (\lm_1-\bt)(2(2\al-1)\lm_1 - \al^2 -2\al\bt+2\bt) = 0.
                \end{equation}
                But $\lm_1-\bt = 0$ in $\sh{4A}$
                and the other factor is not $0$ in $\sh{4A}$.

                From $0 = \lm_2 = \lm_1^{a_0}(a_2) = \lm_1^{a_1}(a_{-1})$
                we deduce that $a_0a_2 = 0 = a_1a_{-1}$.
        \end{proof}

        \begin{lemma}
                \label{lem-4A-1/4}
                $\al - \frac{1}{4}\in J_{\sh{4A}}$.
        \end{lemma}
        \begin{proof}
                From Lemma \ref{lem-4A-dichot}
                we substitute $\lm_1 = \bt,\lm_2 = \sfrac{\al}{2},\al=\sfrac{1}{4},\bt=\sfrac{1}{32}$
                into \eqref{eq-4B-badman}
                and see that this is not a solution.
                Therefore $a_{-2} - a_2\in I_{\sh{4A}}$.
                After substituting similarly,
                we have that $z = -\al\bt a_0 + \frac{1}{2}(\al-\bt)(a_{-1}+a_1) - s$
                is a $0$-eigenvector for $a_0$.
                Then from the fusion rule $0\star0 = \{0\}$,
                \begin{equation}
                        a_0(zz) = -\al(\al-\frac{1}{4})(\al-\bt)(\bt(a_{-1}+a_1)+2s) = 0
                \end{equation}
                and as there can be no further relations in $I_{\sh{4A}}$
                since $\sh{4A}$ is $5$-dimensional,
                we must have $\al-\frac{1}{4}\in J_{\sh{4A}}$.
        \end{proof}

        \begin{theorem}
                \label{thm-4A}
                The algebra $A$, denoted $\sh{4A_\bt}$,
                over $R = \ZZ[\frac{1}{2},\bt]$ for $\bt\neq0,\frac{1}{4},1$,
                with basis $a_{-1},a_0,a_1,a_2,s$
                and multiplication from Table \ref{tbl-4A}
                is a $2$-generated $\Phi(\sfrac{1}{4},\bt)$-axial algebra for all $\bt\in\kk\smallsetminus\{1,0,\frac{1}{4},\frac{1}{8},\frac{1}{16}\}$
                and is the cover of $\sh{4A}$.
        \end{theorem}
        \begin{proof}
                $A$ is generated by $a_0,a_1$,
                so we only need to check the fusion rules for $a_0,a_1$.
                We do this in \cite{gap}.
                This is easily done using Lemma \ref{lem-eigvect-ising}.
        \end{proof}

        \begin{lemma}
                \label{lem-4A-form}
                The algebra $\sh{4A_\bt}$ is Frobenius;
                it is positive-definite for $0 < \bt < \frac{1}{2}$,
                and degenerate for $\bt = \frac{1}{2}$.
        \end{lemma}
        \begin{proof}
                We check by direct computation in \cite{gap}
                that the form on $\sh{4A_\bt}$
                induced as a quotient of the form in Section \ref{sec-mt}
                is Frobenius.
                We also use \cite{gap}
                to calculate the Gram matrix,
                the determinant of which is
                \begin{equation}
                        \frac{\bt}{8}(1-2\bt)^3.
                \end{equation}
                This is positive when $\bt(1-2\bt)$ is positive.
                Its only acceptable root is $\bt = \frac{1}{2}$.
        \end{proof}

        \begin{lemma}
                \label{lem-4X-miyam}
                On $\sh{4A_\bt}$ and $\sh{4B_\al}$,
                $\size{\tau(a_0)\tau(a_1)} = 2$,
                and $\size{\tau(a_0)\tau(a_1)}\leq4$ in any larger algebra.
        \end{lemma}
        \begin{proof}
                $\tau(a_0)$ and $\tau(a_1)$ are the permutation matrices
                corresponding to $(1,3)(2)(4)(5)$ and $(1)(2,4)(3)(5)$,
                respectively, 
                on both $\sh{4A_\bt}$ and $\sh{4B_\al}$.
                By Lemma \ref{lem-global-order},
                the order of $\size{\tau(a_0)\tau(a_1)}$ is bounded by $2+2 = 4$,
                the size of $a_0^T\cup a_1^T$.
        \end{proof}

        \begin{table}[h]
        \begin{center}
        \renewcommand{\arraystretch}{1.2}
        \begin{tabular}{ll}
                \hline Description \quad & Products \\
                \hline
                $a_{-1},a_0,a_1,a_2,s$ &
                        $a_is = -\bt^{2} a_{i} + \frac{1}{8}(1 - 4\bt)\bt (a_{i+1} + a_{i+3}) + \frac{1}{4}(1 - 4\bt) s$ \\
                $a_i = a_{i\bmod4}$ &
                        $ss = \frac{1}{32}\bt(4\bt-1)(8\bt-1) (a_{-1} + a_0 + a_1 + a_2) + \frac{1}{4}\bt(3 - 8\bt) s$ \\
                $\la a_i,a_{i+2}\ra\cong\sh{2B}$ &
                        $(a_i,a_j) = \bt$ \\
                \hline
        \end{tabular}
        \caption{$\sh{4A_\bt}$}
                \label{tbl-4A}
        \end{center}
        \end{table}


\section{The cover of $\sh{3A}$}
        \label{sec-3A}

        We see that \eqref{eq-mother} is not zero in $\sh{3A}$,
        so that we have a relation in the algebra.
        Instead, for the case of the cover of $\sh{3A}$ only,
        we have to start by making an extra assumption.
        Namely, suppose also that $a_3 = a_0$.

        Then $a_{-2} = a_1, a_{-1} = a_2$,
        so that $s_2 = s_1 = s_2^f$
        and $\lm_2 = \lm$.
        By Lemma \ref{lem-odd-quots},
        also $\lm^f = \lm$
        (so this assumption implies the earlier one).

        From Theorem \ref{thm-mt} we therefore deduce a multiplication table
        on $\{ a_0,a_1,a_2,s \}$.

        The eigenvectors of $a_0$ of eigenvalue not $1$ are,
        from Lemma \ref{lem-eigvect-ising}:
        \begin{equation}
                \begin{aligned}
                \label{eq-eigvect-3A}
                & z = ((1-\al)\lm_1-\bt)a_0 + \frac{1}{2}(\al-\bt)(a_1 + a_2) - s, \\
                & x = (\bt-\lm_1)a_0 + \frac{1}{2}\bt(a_1+a_2) + s, \\
                & y = a_1-a_2.
                \end{aligned}
        \end{equation}

        Now $0\star0=\{0\}$ means $zz\in A^{a_0}_0$ and $a_0(zz)=0$.
        We calculate that
        \begin{equation}
                \begin{aligned}
                        & a_0(zz) =
                                -\frac{(\al - \bt)(-4(2\al - 1)\lm + 3\al^{2} + (3\bt - 1)\al - 2\bt)}{4(\al-2\bt)(\al-4\bt)}\cdot \\
                                & \quad \Bigl(
2(-(\al^{2} - 2(\bt + 1)\al - 4\bt^{2} + 4\bt)\lm + \bt\al^{2} - \bt(3\bt + 1)\al - 2\bt^{3} + 2\bt^{2}) a_0 \\
& \quad\quad + \bt(-4(\al - 2\bt)\lm + \al^{2} - 2\bt\al - 4\bt^{2}) a_1
         + \bt\al(\al-4\bt) a_2
+ 2(-2(\al - 2\bt)\lm + \al^{2} - 3\bt\al - 2\bt^{2}) s
                        \Bigr).
                 \end{aligned}
        \end{equation}
        The coefficients in the ring must be $0$
        since there cannot be a nontrivial relation among the spanning set,
        since $\sh{3A}$ is $4$-dimensional.
        Apart from $\al-\bt$,
        the only factor of the coefficient of $a_2$ in the above
        is $-4(2\al - 1)\lm + 3\al^{2} + (3\bt - 1)\al - 2\bt$,
        so this must be $0$,
        and indeed this is a common factor for all the coefficients.
        We obtain a relation satisfied by $\sh{3A}$:
        \begin{equation}
                4(2\al-1)\lm = 3\al^2+\al(3\bt-1)-2\bt.
        \end{equation}
        If $\al=1/2$, then the lefthand side is $0$
        and we find $0 = 3/4+3/2\bt-1/2-2\bt = 1/4-1/2\bt$,
        that is, $\bt = 1/2$,
        which contradicts that $\al\neq\bt$.
        Therefore $\al\neq1/2$
        and we have an expression for $\lm$ in $\al,\bt$.
        After substitution we find Table \ref{tbl-3A}.

        \begin{table}[h]
        \begin{center}
        \renewcommand{\arraystretch}{1.2}
        \begin{tabular}{ll}
                \hline Description \quad & Products \\
                \hline
                $\sh{3A'_{\al,\bt}}$ &
                $a_is = \frac{1}{4(1-2\al)}(3\al^3-5\al^2\bt+8\al\bt^2-4\al^2+7\al\bt-4\bt^2+\al-2\bt) a_i$ \\
                $a_0,a_1,a_2,s$ & $\qquad+ \frac{1}{2}\bt(\al-\bt)(a_{i+1}+a_{i+2}) + (\al-\bt) s$ \\
                $a_i = a_{i\bmod3}$ &
                $ss = \frac{(\al-\bt)}{8(1-2\al)}(3\al^3-13\al^2\bt+16\al\bt^2-4\al^2+11\al\bt-8\bt^2+\al-2\bt)(a_0 + a_1 + a_2)$ \\
                 & $\qquad+ \frac{1}{4(2\al-1)}(9\al^3-27\al^2\bt+12\al\bt^2-6\al^2+13\al\bt-6\bt^2+\al) s$ \\
                 & $(a_i,a_j) = \frac{1}{4}\frac{(3\al^{2} + (3\bt - 1)\al - 2\bt)}{(2\al - 1)}$ \\
                 \hline
        \end{tabular}
        \caption{$\sh{3A'_{\al,\bt}}$}
                \label{tbl-3A}
        \end{center}
        \end{table}


        \begin{theorem}
                \label{thm-3A}
                The algebra $A$ over $R = \ZZ[\sfrac{1}{2},\al,\bt,(1-2\al)^{-1}]$ with basis $a_0,a_1,a_2,s$
                and multiplication table \ref{tbl-3A}
                is a $2$-generated $\Phi$-axial algebra,
                denoted $\sh{3A'_{\al,\bt}}$.
                It is a weak cover of $\sh{3A}$.
        \end{theorem}
        \begin{proof}
                $A$ is generated by $a_0,a_1$,
                so we only need to check the fusion rules for $a_0,a_1$.
                We do this in \cite{gap}.
                For $a_0$ this is easily done using \eqref{eq-eigvect-3A}.
                For $a_1$, this follows from the evident $\Sym(3)$-symmetry
                of the multiplication table with respect to $a_0,a_1,a_2$.
        \end{proof}


        \begin{lemma}
                \label{lem-3A-form}
                The algebra $\sh{3A'_{\al,\bt}}$ is Frobenius;
                it is positive-definite for
                \begin{equation}
                        \label{eq-3A-form}
                        (2\al-1)(3\al-\bt-1)(3\al^2+3\al\bt-\bt-1)(3\al^2+(3\bt-9)\al-2\bt+4)<0,
                \end{equation}
                and degenerate when $(3\al-\bt-1)(3\al^2+3\al\bt-\bt-1)(3\al^2+(3\bt-9)\al-2\bt+4)=0$.
        \end{lemma}
        \begin{proof}
                We check by direct computation in \cite{gap}
                that the form on $\sh{3A'_{\al,\bt}}$
                induced as a quotient of the form in Section \ref{sec-mt}
                is Frobenius.
                We also use \cite{gap}
                to calculate the Gram matrix,
                the determinant of which is
                \begin{equation}
                        \frac{-\al^2}{2^9(2\al-1)^5}
                        (3\al-\bt-1)(3\al^2+3\al\bt-\bt-1)(3\al^2+(3\bt-9)\al-2\bt+4)^3.
                \end{equation}
        \end{proof}

        \begin{lemma}
                \label{lem-3a-miyam}
                We have
                $\size{\tau(a_0)\tau(a_1)} = 3$.
        \end{lemma}
        \begin{proof}
                $\tau(a_0),\tau(a_1)$ are the permutation matrices
                of $(1)(2,3)(4)$ and $(1,3)(2)(4)$ respectively on $\sh{3A'_{\al,\bt}}$.
                Lemma \ref{lem-global-order} implies that
                the order of $\tau(a_0)\tau(a_1)$ is bounded by $3$ everywhere.
        \end{proof}

        \noindent
        Felix Rehren\\
        School of Mathematics\\
        University of Birmingham\\
        B15 2TT, United Kingdom \\
        {\tt rehrenf@maths.bham.ac.uk}

\end{document}